\documentclass[reqno]{amsart}
\usepackage{amsmath}
\usepackage{amssymb}
\usepackage{amsthm}
\usepackage{mathrsfs}
\usepackage{color}

\newcommand{\cm}{\chi _{\mathrm{min}}}
\newcommand{\Ci}{\mathscr{C}}

\newcommand{\R}{\mathbb{R}}
\newcommand{\C}{\mathbb{C}}
\newcommand{\Z}{\mathbb{Z}}

\newcommand{\supp}{\operatorname{supp}}

\newcommand{\vertiii}[1]{{\left\vert\kern-0.25ex\left\vert\kern-0.25ex\left\vert #1 
    \right\vert\kern-0.25ex\right\vert\kern-0.25ex\right\vert}}

\usepackage[initials]{amsrefs}
\usepackage{amsmath,amsthm,amscd}
\usepackage{enumerate}

\usepackage{amssymb}
\usepackage{mathrsfs}
\usepackage{graphicx}

\usepackage{color}
\usepackage{mathtools}
\mathtoolsset{showonlyrefs=true}

\usepackage[T1]{fontenc}

\usepackage{amsmath,amsthm,amssymb,amscd}

\usepackage{amssymb}
\usepackage{mathrsfs}
\usepackage{graphicx}

\theoremstyle{plain}
\newtheorem{thm}{Theorem}

\newtheorem{lem}[thm]{Lemma}

\theoremstyle{definition}
\newtheorem{dfn}[thm]{Definition}

\theoremstyle{remark}
\newtheorem{remark}[thm]{Remark}

\numberwithin{equation}{section}

\date{}
\begin{document}
\title[Spectra of expanding maps on Besov spaces]
{Spectra of expanding maps on Besov spaces
}

\author{Yushi Nakano}
\address[Yushi Nakano]{Faculty of Engineering, Kitami Institute of Technology, Hokkaido,
090-8507, JAPAN}
\email{nakano@mail.kitami-it.ac.jp}

\author{Shota Sakamoto}
\address[Shota Sakamoto]{Graduate School of Human and Environmental Studies,
Kyoto University,
Kyoto, 606-8501, Japan} \email{sakamoto.shota.76r@st.kyoto-u.ac.jp}



\begin{abstract}
A typical approach to analysing statistical properties of expanding maps is to show spectral gaps of associated transfer operators in adapted function spaces. The classical function  spaces for this purpose  are   H\"older spaces and Sobolev spaces. Natural generalisations  of these spaces are Besov spaces, on which we show a spectral gap of transfer operators.
\end{abstract}

\maketitle

\section{Introduction}

Let $M$ be a  compact smooth Riemannian manifold endowed with the normalised Lebesgue measure $\mathrm{Leb} _M$, and  $f: M\to M$  a $\Ci ^r$ expanding map  with $r>1$. 
It is  well known  that individual trajectories of expanding maps  tend to have ``chaotic behaviour''. Therefore,
to analyse statistical properties of the  expanding map $f$,  it is typical to instead  study  how densities  of  points evolve under a so-called transfer operator $\mathcal L_{f,g}$ induced by  $f$ with a given $\Ci ^{\tilde r}$ weight function  $g:M\to \C$ with $0<\tilde r \leq r$ (the precise definition will be given below).
In his celebrated paper~\cite{Ruelle89}, 
Ruelle first showed a spectral gap of the transfer operator of expanding maps on the usual H\"{o}lder space $\Ci ^s (M)$ with $0<s\leq \tilde r$  (when $g$ is   real-valued and strictly positive), 
resulting in the demonstration of 
 the existence of a unique equilibrium state $\mu _{g}$ of $f$ at $g$ and exponential decay of correlation functions of any $\Ci ^s$ observables with respect to $\mu _g$. 
(The existence and uniqueness of equilibrium states for $\Ci ^r$ expanding maps had been proved   earlier in his monograph \cite
{Ruellebook} through a thermodynamic approach.)
Furthermore, the spectral gap of the transfer operator was used to investigate the dynamical zeta function \cite{Ruelle1989, Ruelle1990},  several limit theorems \cite{
AD2001,Gouezel2015} and  strong stochastic stability \cites{BY93, BY94}.
As another (deep) development,  Gundlach and Latushkin \cite{GL2003} obtained  an exact formula of  the essential spectral radius of the transfer operator on $\Ci ^s(M)$ in   thermodynamic expression (see also Remark \ref{rmk:5}). 

Recently, the transfer operator was shown to also have a spectral gap on  the Sobolev space $\mathscr W^{s,p}(M)$ in Baillif and Baladi~\cite{BaillifBaladi} (see also Faure~\cite{Faure} for the spectral gap 
via a semiclassical approach, and  Thomine~\cite{Thomine2011} for a spectral gap of the transfer operator of piecewise expanding maps on $\mathscr W^{s,p}(M)$)
and the ``little H\"{o}lder space'' $\Ci _*^{s}(M)$  in Baladi and Tsujii \cite{BaladiTsujii08b}.
Our goal in this paper is to show a spectral gap of the transfer operator on \emph{Besov spaces} $\mathscr B_{pq}^{s}(M)$, 
which are closely related with the previously-studied function spaces (see Remark \ref{rmk:3b}).
We also refer to \cites{Baladi2017,Baladi2017correction,Baladi2018} and references therein for recent development of Banach spaces adapted to (hyperbolic) dynamical systems.

Our method in the proof is a natural generalisation of  the best technology developed in Tsujii and Baladi \cite{BT2008a}. 
Our result gives an answer to Problem 2.40  in the monograph by Baladi \cite{Baladibook2}.


\subsection{Definitions and results}
Before precisely stating our main result,  we introduce some notation.
Recall that $M$ is a compact smooth Riemannian manifold and $f: M\rightarrow M$ is of class $\mathscr{C}^r$ with $r>1$. Let $f$ be an  \emph{expanding map}, i.e., there exist constants $C>0$ and $\lambda _0 >1$ such that $\vert Df^n(x) v\vert \geq C\lambda _0 ^n \vert v\vert $ for each $x\in M$ and $v\in T_xM$. 
Let us set the minimal Lyapunov exponent
\[
\cm =\lim _{n\to \infty}\frac{1}{n}\log 
\inf_{x\in M}\color{black}\inf _{\stackrel{v\in T_xM}{\vert v\vert =1}}  \vert Df^n(x)v\vert .
\]
Then  $\cm>0$ when $f$ is an expanding map. (For the properties of expanding maps, the reader is referred to \cite{Ruellebook, KH95, FU2010} e.g.) 

  Let $g$ be a complex-valued $\mathscr C^{\tilde r}$ function on $M$ with $0<\tilde r\leq r$. 
We assume a technical condition
\begin{equation}\label{eq:ta}
\tilde r\geq 1 \quad \text{or} \quad \tilde r \leq r-1.
\end{equation}
Note that \eqref{eq:ta}  is satisfied for an important application $g=\vert \det Df\vert ^{-1}$, and that if $r\geq 2$, then the condition \eqref{eq:ta} always holds 
(see also Remark \ref{rmk:5} for the condition).
With the notation $g^{(n)}(x) = \prod _{j=0}^{n-1}g(f^j(x))$,  we  set
\begin{align}\label{def: R(g)}
 R(g)= \lim_{n\rightarrow \infty} R_n(g) ^{1/n}, \quad
 R_n(g)=\Vert g^{(n)}  \det Df^n \Vert _{L^\infty} \quad \text{($n\geq 1$)}.
\end{align}
For each $0\leq s\leq \tilde r$, the 
transfer operator $\mathcal  L_{f,g}:\mathscr C^s(M) \to \mathscr C^s(M)$ of $f$ with a $\Ci^{\tilde r}$ weight function $g$ is defined by
\[
\mathcal  L_{f,g}u (x) = \sum _{f(y)=x} g(y) \cdot u (y)  ,\quad u \in \mathscr C^s(M).
\]
%
Standard references for transfer operators are \cites{Baladi, Baladibook2}. 
\begin{remark}\label{rmk:3c}
Denote by $r(A\vert _E)$ and $r_{\mathrm{ess} } (A\vert _E)$ the spectral radius and the essential spectral radius  of a bounded operator $A: E\to E$ on a Banach space $E$, respectively.
Due to Ruelle \cite{Ruelle89}, $r(\mathcal L_{f,g} \vert _{\mathscr C^s(M)})$ is  bounded by $\exp P_{\mathrm{top}} (\log \vert g\vert )$, and is equal to $\exp P_{\mathrm{top}} (\log g)$ when $g$ is real-valued and strictly positive (i.e., $\inf _{x\in M} g >0$), 
where   $P_{\mathrm{top}}(\phi)$ is the topological pressure of a continuous function $\phi :M\to \mathbb R$ 
(refer to \cite{Walters1982} for the definition of topological pressure).
It also follows from \cite[Lemma 2.16]{Baladibook2} that $R(g) \geq \exp P_{\mathrm{top}} (\log \vert g\vert )$.

Furthermore, in view of   \cite[Lemma A.1]{BT2008a} regarding coincidence of eigenvalues of abstract linear operators in different Banach spaces outside of the essential spectral radii, 
we get the following: 
If  $r_{\mathrm{ess}}(\mathcal L_{f,g} \vert _{\mathscr B^s_{pq}(M)}) < \exp P_{\mathrm{top}} (\log g)$ with  real-valued and strictly positive $g$,
then $r(\mathcal L_{f,g}\vert _{\mathscr B^s_{pq}(M)} )=\exp P_{\mathrm{top}} (\log g)$, where   $\mathscr B^s_{pq}(M)$ is the Besov space on $M$ (see below for definition).
In particular,  the transfer operator on $\mathscr B^s_{pq}(M)$ is \emph{quasi-compact} (has a spectral gap), which is our goal in this paper.
\end{remark}
$\mathcal  L_{f,g}$ can be extended to a bounded operator on $L^p(M)$ with $p\in [1,\infty]$. 
 (Since some Besov spaces do not coincide with the completion of $\Ci^s(M)$ with respect to its Besov norm, this extension would be necessary; see \cite[\S A.1]{Taylor}.) 
 Indeed, a change of variables shows
\begin{equation}\label{eq:tobasic}
\int \mathcal  L_{f,g} u \cdot \varphi d\mathrm{Leb} _M = \int u\cdot \varphi \circ f \cdot g\cdot \vert \det Df \vert d\mathrm{Leb} _M,
\end{equation}
for any $u\in L^\infty(M)$ and  $\varphi \in L^1(M)$. 
On the other hand, the operator 
$\varphi \mapsto \varphi \circ f \cdot g\cdot \vert \det Df\vert$  is bounded  on $L^{p^\prime}(M)$ with $1/p^\prime+1/p=1$: notice  that 
$
\Vert  \varphi \circ f \cdot g \det Df\Vert _{L^{p^\prime}}\leq R_1(g) \Vert \varphi \circ f\Vert _{L^{p'}}$ and that in the case $1\leq p' <\infty$,
\[
\Vert \varphi \circ f\Vert _{L^{p'}} = \left(\int \mathcal L_{f,\vert \det Df\vert ^{-1}} 1 \cdot \vert \varphi \vert ^{p'} d\mathrm{Leb} _M\right)^{1/p'}
\leq \Vert \mathcal L_{f,\vert \det Df\vert ^{-1}} 1\Vert _{L^{\infty}} ^{1/p'} \Vert \varphi \Vert _{L^{p'}}
\]
by virtue of \eqref{eq:tobasic}, while $\Vert \varphi \circ f\Vert _{L^{\infty}}\leq \Vert \varphi \Vert _{L^{\infty}}$.
 Therefore,  the transfer operator has a continuous extension to $L^p(M)$ by the duality \eqref{eq:tobasic}. 
The extension will also be denoted by $\mathcal L_{f,g}$.

Below we  define  Besov spaces associated with 
 a partition of unity of $M$.
We first recall the definition of  the Besov spaces on $\mathbb{R}^d$,
where $d$ is the dimension of $M$.
Let $\rho:\R\to\R$
 be a  $\Ci^\infty$ function such that 
\begin{align}\label{def: rho}
0\le \rho \le 1,\quad
\rho(t)=
\begin{cases}
1 \quad (t \le 1),\\
0 \quad (t \ge 2).
\end{cases}
\end{align}
For each nonnegative integer $n$, define 
radial functions $\psi_n \in \Ci ^\infty_0(\mathbb{R}^d, \R)$   by
\begin{align}\label{eq:0603b}
\psi_n(\xi)=
\begin{cases}
\rho (\vert \xi \vert) \quad &(n=0),\\
\rho(2^{-n}\vert \xi \vert)-\rho(2^{-n+1}\vert \xi \vert) \quad &(n\geq 1).
\end{cases}
\end{align}
 For a tempered distribution $u$ (that is, $u$ is in the dual space of the set of rapidly decreasing test functions), the operator $\Delta_n$ is given by $\Delta_n u= \mathcal{F}^{-1}[\psi_n \mathcal{F}u]$ with $n\geq 0$,
 where $\mathcal{F}$ is the Fourier transform. 
 Then we have $\sum_{n\ge 0} \Delta_n u =u$, called the Littlewood-Paley (dyadic) decomposition. 
 This decomposition was first employed in context of dynamical systems theory to analyse the spectra of transfer operators of Anosov diffeomorphisms by Baladi and Tsujii \cite{BaladiTsujii07}. They also applied the decomposition to spectral analysis of transfer operators of 
  expanding maps in the  little H\"{o}lder space $\Ci _*^{s}(M)$ in  a survey \cite[Subsection 3.2]{BaladiTsujii08b}, see Remark \ref{rmk:3b} for the definition of $\Ci ^s_*(M)$.

For $s\in \mathbb{R}$, $p,q\in [1, \infty]$, we define the Besov space $B^s_{pq}(\mathbb{R}^d)$
 as a set of tempered distributions 
 $u$ on $\mathbb{R}^d$ whose norm 
\begin{equation}
\Vert u\Vert_{B^s_{pq}} =
\begin{cases}
\Big( \sum_{n\ge 0} 2^{sqn}\Vert \Delta_n u\Vert_{L^p}^q\Big)^{1/q} \quad &(q<\infty),\\
\sup_{n\ge 0} 2^{sn}\Vert \Delta_n u\Vert_{L^p} \quad &(q=\infty)
\end{cases}
\end{equation}
is finite.
We remark that $u \in B^s_{pq}(\mathbb{R}^d)$ if and only if there are a constant $C>0$ and a nonnegative sequence $\{ c_n \} _{n\geq 0} \in \ell^q$ with $\Vert \{ c_n \} _{n\geq 0}\Vert_{\ell^q}\le 1$ such that 
\begin{equation}\label{eq:above3}
\Vert \Delta_n u\Vert_{L^p} \le C2^{-ns}c_n
\end{equation} 
for all  $n\geq 0$. 
Once we know $u \in B^s_{pq}(\mathbb{R}^d)$, we can take $C=\Vert u\Vert_{B^s_{pq}}$.

Since $M$ is compact, there are a finite open covering $\{V_i\}_{i=1}^I$ and a system of local charts $\{\kappa _i\}_{i=1}^I$ such that $\kappa_i: V_i \rightarrow U_i$ is of class $\Ci ^\infty$  with an open set  $U_i\subset \mathbb{R}^d$  for each $1\leq i\leq I$. Let $\{\phi_i\}_{i=1}^I$  be a family of  $\Ci ^\infty$ functions such that $\{ \phi _i\circ \kappa _i\}_{i=1}^I$ is a partition of unity of $M$ subordinate to the covering $\{V_i\}_{i=1}^I$, that is,    $\phi_i\circ \kappa _i$ is a $\Ci ^\infty$ function on $M$ with values in $[0,1] \subset \R $ whose  support   is contained in $V_i$ for each $1\leq i\leq I$ and $\sum _{i=1}^I \phi _i \circ \kappa _i(x) =1$ for all $x\in M$.
In this paper, the  support of a continuous function $\phi :M \to \R$  is defined as  the closure of $\{x \in M \mid  \phi (x) \neq 0\}$. 
\begin{dfn}
The Besov space $\mathscr B^s_{pq}(M)$  
on $M$ is the space of tempered distributions 
$u$ on $M$ whose norm
\begin{equation}\label{eq:mm1}
\Vert u\Vert _{\mathscr B^s_{pq}} = \sum _{1\leq i\leq I} \Vert \phi _i\cdot u \circ \kappa _i^{-1}\Vert _{B^s _{pq}} 
\end{equation}
is finite. 
This definition does not depend on the choice of charts or the partition of unity, see \cite{Triebel86}. 
\end{dfn}

Let $U \subset \mathbb{R}^d$ be a nonempty bounded open set.
We write 
$\Ci ^s(U)$, $L^p(U)$ and  $B^s_{pq}(U)$  for the subspace of $\Ci ^s(\R^d)$, $L^p(\R^d)$ and $B^s_{pq}(\R ^d)$, respectively, such that  the support  of each element in these spaces is included in $U$. 
Then we have $B^s_{pq}(U) \subset L^p(U)$ for $s>0$ and $p$, $q\in [1,\infty]$.\footnote{
Indeed, by H\"{o}lder's inequality, for $1\le q\le \infty$ it holds 
\begin{align*}
\Vert u\Vert_{L^p} &= \big\Vert \sum_{n\ge 0} \Delta_n u\big\Vert_{L^p}\le \sum_{n\ge 0} \Vert \Delta_n u\Vert_{L^p}\\
&\le \Big( \sum_{n\ge 0} 2^{sqn} \Vert \Delta_n u\Vert_{L^p}^q\Big)^{1/q} \Big( \sum_{n\ge 0} 2^{-sq' n}\Big)^{1/q'} = \big(1-2^{-sq'} \big)^{-1/q'} \Vert u\Vert_{B^s_{pq}},
\end{align*}
where $1/q+1/q'=1$ with a usual modification for $q=\infty$.
}
Hence,  from the argument following \eqref{eq:tobasic} it holds that $\mathcal L_{f,g} u$ is  in $L^p(M)$ for each  $u\in \mathscr B^s_{pq}(M)$  with $s>0$ and $p,q\in [1,\infty]$. 
In particular, $\mathcal L_{f,g} u$ is a tempered distribution and $\Vert \mathcal L_{f,g} u\Vert _{\mathscr B^s_{pq}}$  is well-defined (but possibly takes $+\infty$).

Here we provide our main theorem for an upper bound of the essential spectral radius of the transfer operator on $\mathscr B^s_{pq}(M)$, which concludes (due to   Remark \ref{rmk:3c}) a spectral gap of the transfer operator on $\mathscr B^s_{pq}(M)$ when   $g$ is real-valued and strictly positive and $s$ is sufficiently large:
\begin{thm}\label{thm: main}
Let $f :M \to M$ be a $\Ci^r$  expanding map  
with $r>1$, and $g:M\to \C$  a $\Ci^{\tilde r}$ function with $0<\tilde r\le r$ satisfying \eqref{eq:ta}. Let $s\in (0, \tilde r]$ and $p,q \in [1,\infty]$.
In addition, when $q<\infty$, assume that $s$ is strictly smaller than $\tilde r$. 
Then 
$\mathcal L_{f,g}$ can be extended to a bounded operator on $\mathscr B^s_{pq}(M)$, 
and it holds
\begin{equation}
r_{\mathrm{ess}} (\mathcal L _{f, g} \vert _{\mathscr B^s_{pq}(M)} ) \leq \exp (-s \chi _{\mathrm{min}}) \cdot R(g),
\end{equation}
where $R(g)$ is defined in \eqref{def: R(g)}.
\end{thm}

\begin{remark}\label{rmk:3b}
The Besov spaces can represent many other  function spaces. As an important example, $B^s_{\infty\infty}(\mathbb{R}^d)$  is known to coincide with the H\"{o}lder space $C ^s(\mathbb{R}^d)$ when $s$ is not an integer (whereas $B^s_{\infty \infty}(\R ^d)$ strictly includes $C^s (\R^d)$ when $s$ is an integer), see \cite[\S A.1]{Taylor}. 
Hence Theorem \ref{thm: main}  is a generalisation  of a well-known result by Ruelle \cite[Theorem 3.2]{Ruelle89}. 
We note that the  \emph{little H\"{o}lder space} $\Ci _*^{s}(M)$ given in \cite{BaladiTsujii08b} is defined as the completion  of $\mathscr C^\infty (M)$ for the norm $\Vert \cdot \Vert _{\mathscr B^s_{\infty \infty}}$ (which is smaller than $\mathscr C^s (M)$, see \cite[\S A.1]{Taylor}),
so that one can see that the proof of Theorem \ref{thm: main} can be translated literally to the case when the functional space where the transfer operator acts is $\Ci _*^s(M)$ (this was essentially done in \cite[Theorem 3.1]{BaladiTsujii08b} with $\tilde r =r-1$).

Furthermore, for   the Sobolev  space $W^{s,p}(\mathbb{R}^d)$ with $s \in \mathbb{R}$ and $p \in (1,2]$, it holds $B^s_{pp}(\mathbb{R}^d)\subset W^{s,p}(\mathbb{R}^d)\subset B^s_{p2}(\mathbb{R}^d)$ and $B^s_{p'2}(\mathbb{R}^d)\subset W^{s,p'}(\mathbb{R}^d)\subset B^s_{p'p'}(\mathbb{R}^d)$, where $1/p+1/p' =1$  (if $p\neq 2$, these inclusions are strict; see \cite[p.161]{Stein70} for the cases when $s\in \mathbb{N}$). 
These inclusions imply $\mathscr W^{s,2}(M) =\mathscr B^s_{22}(M)$, thus Theorem \ref{thm: main} recovers a part of the result by Baillif and Baladi \cite{BaillifBaladi}.

Other famous functional spaces constructed by using the Littlewood-Paley decomposition are Triebel-Lizorkin spaces $F^s_{pq}(\R ^d)$. 
It seems possible  to prove Theorem \ref{thm: main} with $B^s_{pq}(\R ^d)$ replaced by $ F^s_{pq}(\R ^d)$ in the same manner as in the proof of Theorem \ref{thm: main}. 
We note that $W^{s,p}(\R ^d)=F^s_{p2}(\R) $ for each $1<p<\infty$, see \cite{Triebel1973}.
\end{remark}

\begin{remark}\label{rmk:6}
Theorem \ref{thm: main} is far from optimal: Gundlach and Latushkin \cite{GL2003} showed that $r_{\mathrm{ess}} (\mathcal L_{f,g} \vert _{\Ci ^s(M)})$ is \emph{equal to}
\begin{equation}\label{eq:rev4}
\sup _{\mu \in \mathrm{Erg}  (f)} \exp  \left( -s \chi _ \mu ^- + h_\mu  + \int \log \vert g\vert d\mu  \right) ,
\end{equation}
where $\mathrm{Erg} (f)$ is the set of ergodic $f$-invariant probability measures, $\chi _ \mu ^-$ is the \emph{smallest Lyapunov exponent} of $\mu$ (for  $Df$ over $f$), and $h_\mu$ is the \emph{Kolmogorov-Sinai entropy} of $\mu$ (over $f$).  We refer to   \cites{Walters1982,Viana2014} for the definitions of Lyapunov exponents and  Kolmogorov-Sinai entropy.
By the variation principle for topological pressure (see \cite[\S 9]{Walters1982} e.g.), we have 
\[
P_{\mathrm{top} } (\phi ) = \sup _{\mu \in \mathrm{Erg}  (f)}   \left(  h_\mu  + \int \phi  d\mu  \right) ,
\]
so that, since $\chi ^-_\mu \geq \chi _{\mathrm{min}}$ for any $\mu \in \mathrm{Erg} (f)$, 
\eqref{eq:rev4} is bounded by $\exp ( -s \chi _{\mathrm{min}} + P_{\mathrm{top}} (\log \vert g\vert ))$. Therefore, when $g$ is real-valued and strictly positive,  the transfer operator on $\Ci ^s(M)$ is quasi-compact for \emph{any} $0<s \leq \tilde r$, while our $s$ should be large to deduce the quasi-compactness from Theorem \ref{thm: main}.


We also give a remark on an exceptional case $g=\vert \det Df \vert ^{-1}$: If $g=\vert \det Df \vert ^{-1}$, then $R(g) =1$ and 
it follows from the Ruelle inequality and the Pesin identity (see e.g.~\cite{Mane2012}) that $\exp P_{\mathrm{top}} (\log \vert g \vert ) =1$.
Therefore, together with Remark \ref{rmk:3c}, Theorem \ref{thm: main} implies the quasi-compactness of $\mathcal L_{f,g}$ 
for any $0<s<\tilde r$.
\end{remark}

\begin{remark}\label{rmk:5}
It is a natural question whether one can get a Gundlach-Latushkin type formula of the essential spectral radius   on  $\mathscr B^{s}_{pq}(M)$, but we think that our method  does not work for this purpose. 
Roughly speaking, we avoided  the problem of exponential growth of the number of inverse branches of $f^n$  by a duality argument (see e.g.~\eqref{eq:rev7}), and  this made our proof a little simpler. 
Its main shortcoming is that it does not seem that one can use thermodynamic techniques in estimating the essential spectral radius at the final step of the proof, which might be essential to obtain a  Gundlach-Latushkin type formula.
Baladi \cite[Theorem 2.15]{Baladibook2} recently  showed a Gundlach-Latushkin type upper bound of the essential spectral radius of the transfer operator of expanding maps on the Sobolev space $\mathscr W^{s,p}(M)$ by using the  Littlewood-Paley decomposition, which may be helpful to answering this question.

Another shortcoming of our approach is the technical condition \eqref{eq:ta}  for the regularity of weight functions. 
This problem is also solved in the Sobolev case in the approach of \cite[Theorem 2.15]{Baladibook2} (see Remark \ref{rmk:8} also), and
 it is highly likely that one can remove the condition \eqref{eq:ta}  by a similar argument.
\end{remark}
%
%
%

\section{Proof of Theorem \ref{thm: main}}\label{section:proof}

\subsection{Transfer operators on $\R ^d$} 
We shall  deduce Theorem \ref{thm: main} from  a theorem for  transfer operators on Besov spaces of compactly  supported functions on $\R ^d$. 
Let $U$ and  $U^\prime$ be nonempty bounded open  subsets  of $\R^d$. 
Let $F$ be   a $\Ci ^r$ mapping on $\R ^d$ with $r>1$ such that $U'\cap F^{-1}(U) \neq \emptyset$ and
\begin{align}\label{def: lambda}
\lambda =  \min _{x\in U'\cap F^{-1}(U)} \min _{\stackrel{v\in \R^d}{\vert v\vert =1}}\left\vert DF(x) v\right\vert 
\end{align}
is larger than $1$.
Let $G$ be a $\Ci ^{\tilde r}$ function whose support is included in $U'\cap F ^{-1}(U)$.
Furthermore, we assume that there are finitely many compact subsets $\{ K_j\}_{j=1}^{N(F)}$ of $\R ^d$ such that 
\begin{equation}\label{eq:re7}
\text{$\mathrm{int} (K_i) \cap \mathrm{int} (K_j) =\emptyset$ when $i\neq j$,} \quad \text{and} \quad  \text{$\supp G = \cup _{j=1}^{N(F)} K_j$},
\end{equation} 
and for each $1\leq j \leq N(F)$ one can find a small neighbourhood $V_j$ of $K_j$ satisfying  that 
\begin{equation}\label{eq:re8}
\text{$F : V_j \to F(V_j)$  is a $\Ci ^{r}$ diffeomorphism.}
\end{equation}

Define a bounded operator $\mathbb L_{F,G}: \Ci ^s(\R ^d)\to \Ci ^s (U)$ by
\[
\mathbb L_{F,G}u(x) =
 \sum _{F(y)=x} G(y) \cdot u(y).
\] 
Then, in a  manner similar to one below \eqref{eq:tobasic}, 
  for each $p\in [1,\infty]$ we can extend $\mathbb L_{F,G}$ to a bounded operator  from $L^p(\R ^d)$ to $L^p(U)$. 
The only essential difference from the argument below \eqref{eq:tobasic} is  that in the case $1<p\leq \infty$, we need to notice 
\begin{align}\label{eq:rev7}
\Vert \varphi \circ F\cdot G\det DF\Vert _{L^{p'}}
&=\Vert 1_{U'\cap F^{-1}(U)}\cdot \varphi \circ F\cdot  G\det DF \Vert _{L^{p'}}\\\notag
&\leq \Vert  G \det DF  \Vert _{L^{\infty}} \Vert 1_{U'\cap F^{-1}(U)}\cdot \varphi \circ F \Vert _{L^{p'}} 
\end{align}
 with $1/p'+1/p=1$, where $1_{U'\cap F^{-1}(U)}$ is the indicator function of $U'\cap F^{-1}(U)$.
Furthermore, it follows from this argument that the operator norm of $\mathbb L_{F,G}: L^p(\R ^d)\to L^p (U)$ is bounded by a nonnegative number $ \alpha \equiv \alpha (F,G,p)$, given by
\begin{align}\label{def: alpha}
\quad\alpha = 
\begin{cases}
{\displaystyle\sup _{x\in F(U') \cap U} \left\vert \sum _{F(y)=(x )}  \frac{1}{\vert \det DF (y)\vert } \right\vert  ^{1/p'}}\left\Vert G \det DF\right\Vert _{L^\infty} & (1< p\leq \infty),\\
\left\Vert G \det DF\right\Vert _{L^\infty} & (p =1)
\end{cases}
\end{align}
with $1/p'+1/p=1$.

We define $C_1$ as the operator norm of $\Delta _n$ on $L^p(\R ^d)$,
\begin{align}\label{def: C_1}
C_1 = \sup_{\Vert u\Vert_{L^p}=1} \Vert \Delta_n u\Vert_{L^p}
\end{align}
(such $C_1$ is bounded and independent of $n$: see \cite[Remark 2.11]{BCD}).
Notice also that if we let $\varphi _0 \in \Ci ^\infty (U)$ such that $\varphi _0 \circ F \equiv 1$ on the support of $G$, then $\mathbb M : u \mapsto u\cdot \varphi _0$ is a bounded operator from  $B^{s}_{pq}(\R ^d)$ to $B^{s}_{pq}(U)$. Denote the operator norm of $\mathbb M$ by $\tilde C=\tilde C(s,p,q,\varphi _0)>0$.
For $s>0$, let
\begin{align}\label{def: gamma_s}
\tilde \gamma _s=C_1 2^{5s}(1-2^{-s})^{-1}\ \text{and}\ \gamma _s = \tilde C \tilde \gamma _s.
\end{align} 
Theorem \ref{thm: main} will follow from the next Lasota-Yorke type inequality. 
\begin{thm}\label{thm: decomposition}
Assume that $(s,p,q)$ satisfies the condition in Theorem \ref{thm: main}, and let $F$ and $G$ be as above. 
Then, $\mathbb L_{F,G}$ can be extended to a bounded operator from $B^s_{pq}(U')$ to $B^s_{pq}(U)$. 
Furthermore, there are a constant $\hat C_{F,G} >0$ and $\sigma \in (0,s)$ such that
\[
\Vert \mathbb L_{F,G} u \Vert _{B^s_{pq}} \leq \gamma _s\alpha  \lambda ^{-s}  \Vert u\Vert _{B^s_{pq}} + \hat C_{F,G} \Vert u \Vert _{B^\sigma _{pq}}
\]  
for all $u\in B^s_{pq}(U')$, 
where the constants are defined in \eqref{def: lambda}, \eqref{def: alpha}, and \eqref{def: gamma_s}.
\end{thm}

\begin{remark}\label{rmk:8}
The proof of Theorem \ref{thm: decomposition} works even when $\lambda \leq 1$, although it may give no spectral information for the original dynamical system $f$ (in our setting \eqref{eq:ta} about $\tilde r$): the role of $F$ will played by $f^n$ in local chart and $\lambda \leq 1$ corresponds to $\chi _{\mathrm{min}} \leq 0$, so that Theorem \ref{thm: main} only means $r_{\mathrm{ess}} ( \mathcal L_{f,g} \vert _{\mathscr B^s_{pq} (M)})$ is bounded by $\exp (-s\chi _{\mathrm{min}}) \cdot R(g)$ with $\exp (-s\chi _{\mathrm{min}}) \geq 1$ and $R(g) \geq \exp P_{\mathrm{top}} (\log \vert g\vert )$ (see Remark \ref{rmk:3c}).
However yet, Theorem \ref{thm: decomposition} with $F=\mathrm{id}$ ($\lambda =1$) seems to be actually helpful to demonstrate a spectral gap 
with $\tilde r$ in full generality (that is, with $\tilde r$ being \emph{not} in the range of \eqref{eq:ta}), see \cite[Theorem 2.15]{Baladibook2} for the Sobolev case.
\end{remark}

\begin{proof}
We follow  \cite{BaladiTsujii08b} and \cite{Baladibook2}.
Let $u\in B^s_{pq}(U')$. 
It follows from \eqref{eq:above3} that  both $ \sum_{\ell:\ell\hookrightarrow n}\Delta _\ell u$ and $ \sum_{\ell:\ell \not \hookrightarrow n}\Delta _\ell u$ are in $L^p(\mathbb R^d)$ for any $n\geq 0$. 
Thus, if we define $\mathbb L_{0,n} u$ and $\mathbb L_{1,n} u$ with $n\geq 0$ by
\begin{align*}
 \mathbb L_{0,n}' u&= \sum_{\ell:\ell\hookrightarrow n} \Delta _n\mathbb{L}_{F,G}\Delta _\ell u,\quad   \mathbb L_{0,n} u = \mathbb M\circ \mathbb L_{0,n}' u,\\
   \mathbb L_{1,n} 'u &= \sum_{\ell:\ell\not\hookrightarrow n}  \Delta _n \mathbb{L}_{F,G} \Delta _\ell u , \quad   \mathbb L_{1,n} u = \mathbb M\circ \mathbb L_{1,n}' u,
\end{align*}
then both $\mathbb L_{0,n} u$ and $\mathbb L_{1,n} u$
are  in $L^p(U)$ (due to  \eqref{def: alpha} and \eqref{def: C_1}), 
where  the  summations are taken over nonnegative integers $\ell $ and we write $\ell \hookrightarrow n$ if $2^n \le \lambda^{-1} 2^{\ell +4}$ and $\ell \not\hookrightarrow n$ if $2^n > \lambda^{-1} 2^{\ell +4}$.
It is obvious from the fact $\sum _{\ell \geq 0} \Delta _\ell u =u$ that
\[
\Delta _n \mathbb L_{F,G} u =  \mathbb L_{0,n} u  +  \mathbb L_{1,n} u  ,
\]
so we have by Minkowski's inequality that 
\[
\Vert \mathbb L_{F, G} u \Vert _{B^s_{pq}} \leq \big( \sum _{n\geq 0} 2^{sqn} \Vert \mathbb L_{0,n} u  \Vert _{L^p} ^q \big) ^{1/q} +  \big( \sum _{n\geq 0} 2^{sqn} \Vert \mathbb L_{1,n} u  \Vert _{L^p} ^q \big) ^{1/q}
\]
in the case $q<\infty$ (a corresponding inequality also holds for $q=\infty$).
Therefore, it suffices to show the following two inequalities for each $q<\infty$ (and corresponding two inequalities for $q=\infty$):
\begin{equation}\label{eq:0602}
 \Big( \sum _{n\geq 0} 2^{sqn} \Vert \mathbb L_{0,n} 'u  \Vert _{L^p} ^q \Big) ^{1/q} 
 \leq \tilde \gamma _s \alpha \lambda ^{-s} \Vert u \Vert _{B^s_{pq}},
\end{equation}
and
\begin{equation}\label{eq:0602b}
\Big( \sum _{n\geq 0} 2^{sqn} \Vert \mathbb L_{1,n} 'u  \Vert _{L^p} ^q \Big) ^{1/q} \leq \bar C_{F,G} \Vert u \Vert _{B^\sigma _{pq}}
\end{equation}
with some constant $\bar C_{F,G} >0$ and $\sigma \in (0,s)$.

 
First we show \eqref{eq:0602}. We focus on the case when $q$ is finite, but the other case is analogous. 
Using \eqref{eq:above3}, \eqref{def: alpha} and \eqref{def: C_1}, we can find a nonnegative $\ell ^q$ sequence $\{c_\ell \}_{\ell \geq 0}$ satisfying $\Vert \{c_\ell \}_{\ell \geq 0}\Vert _{\ell ^q}\leq 1$ such  that  
\begin{multline}\label{eq:above22}
 \Big( \sum _{n\geq 0} 2^{sqn} \Vert \mathbb L_{0,n} 'u  \Vert _{L^p} ^q \Big) ^{1/q}  \leq  C_1 \alpha \Vert u\Vert_{B^s_{pq}} \Big( \sum_{n\ge 0} 2^{nsq}\Big( \sum_{\ell : \ell\hookrightarrow n}2^{-s\ell} c_\ell \Big)^q\Big)^{1/q}
 \\
  \leq C_1\alpha  \Vert u\Vert_{B^s_{pq}}  \Big( \sum_{n\geq 0} \Big(\sum_{\ell\ge 0}2^{s(n-\ell)} 1_{\{0\hookrightarrow (\cdot )\}}(n-\ell)c_\ell\Big)^q\Big)^{1/q},
\end{multline}
where $1_{\{ 0\hookrightarrow (\cdot )\}}$ is the indicator  function of $\{ 0\hookrightarrow (\cdot )\}=\{ \ell \in \Z \mid 0 \hookrightarrow \ell\}$.

To estimate it, we recall Young's inequality for the locally compact group $\Z$.
We define $\ell ^q(\Z )$ as the set of   (two-sided) nonnegative sequences $a_{(\cdot)} \equiv \{ a_\ell \}_{\ell \in \Z}$ such 
that $\Vert a_{(\cdot )} \Vert _{\ell ^q(\mathbb Z)} := \left(\sum _{\ell \in \Z} a_\ell ^q\right)^{1/q}$ is finite. Then, it follows from \cite[Lemma 1.4]{BCD} that $\Vert a _{(\cdot )}* b_{(\cdot)} \Vert _{\ell ^q(\Z)} \leq \Vert a _{(\cdot)} \Vert _{\ell ^1(\Z)} \Vert b _{(\cdot)} \Vert _{\ell ^q(\Z)}$ for all $a_{(\cdot)} \in \ell ^1(\Z)$ and $b_{(\cdot )}\in \ell ^q(\Z)$, where $a _{(\cdot )}* b_{(\cdot)} $ is the convolution of $a _{(\cdot )}$ and $b_{(\cdot)} $ given by $a _{(\cdot )}* b_{(\cdot)} (n ) = \sum _{\ell \in \mathbb Z} a_{n-\ell} b_{\ell}$ for $n\in \mathbb Z$.
Thus, if we let $\{ \tilde c_\ell \} _{\ell \in \Z}$ be as $\tilde c_\ell = c_\ell$ for $\ell \geq 0$ and $=0$ for $\ell <0$, then we have
\begin{multline}\label{eq: geological sequence}
  \sum_{n\ge 0} \Big(\sum_{\ell\ge 0}2^{s(n-\ell)} 1_{\{0\hookrightarrow (\cdot )\}}(n-\ell)c_\ell\Big)^q \\
  \leq 
 \sum_{n\in \Z} \Big(\sum_{\ell\in \Z}2^{s(n-\ell)} 1_{\{0\hookrightarrow (\cdot )\}}(n-\ell)\tilde c_\ell\Big)^q 
 =  \left\Vert \left(2^{s(\cdot)}\cdot  1_{\{0\hookrightarrow (\cdot )\}} \right) * \tilde c_{(\cdot)}  \right\Vert_{\ell^q(\Z)}^q\\
\le \left \Vert 2^{s(\cdot)}\cdot  1_{\{0\hookrightarrow (\cdot )\}} \right\Vert_{\ell^1(\Z)}^q \left\Vert \tilde  c_{(\cdot)} \right \Vert_{\ell^q(\Z)}^q
 \le  \left(\sum_{\ell \in \{ 0\hookrightarrow (\cdot ) \}} 2^{s\ell} \right)^q.
\end{multline}
On the other hand, since
\begin{equation}\label{eq:0602c}
\{ 0\hookrightarrow (\cdot) \}  =\{  \ell \in \Z \mid  \ell \leq 4-\log_2\lambda \},
\end{equation}
we have that
\begin{equation}\label{eq: geological sequence2}
\sum_{\ell \in \{ 0\hookrightarrow (\cdot ) \}} 2^{s\ell} \leq 2^{s(5-\log _2 \lambda )} \sum _{\ell \geq 0} 2^{-s\ell} =\frac{2^{5s}}{1-2^{-s}}\lambda^{-s}.
\end{equation}
By \eqref{eq:above22}, \eqref{eq: geological sequence} and \eqref{eq: geological sequence2}, 
 we obtain \eqref{eq:0602}.

Next, we shall show \eqref{eq:0602b}. 
%
Let $\{\tilde{\psi}_\ell \}_{\ell \geq 1}$ be a family of smooth functions on $\R ^d$ given by
\begin{equation*}
\tilde{\psi}_{n}(\xi) =  
\begin{cases} 
\rho (2^{-1}\vert \xi\vert)   \quad &(n=0),\\
\rho (2^{-n-1}\vert \xi\vert) -\rho (2^{-n+2}\vert \xi\vert) \quad & (n\ge 1)
 \end{cases}
\end{equation*}
for $\xi \in \R^d$, where $\rho$ is defined in \eqref{def: rho}.
Notice that 
 $\tilde{\psi}_n \equiv 1$ on the support  of  $\psi_n$ for each $n \geq 0$.
For a tempered distribution $u$, we define $\tilde \Delta _nu$ by $\tilde \Delta _nu= \mathcal{F}^{-1}[\tilde \psi_n \mathcal{F}u]$ for each $n\geq 0$.
Then it is straightforward to see that 
\[
\mathbb L_{1,n}'u= \sum_{\ell:\ell\not\hookrightarrow n} \Delta _n \mathbb{L}_{F,G} \tilde \Delta _\ell \Delta _\ell u \quad \text{for $u\in L^p(U')$.}
\]

We borrow the following crucial lemma by Baladi \cite{Baladibook2} (from the proof of Lemma 2.21 of \cite{Baladibook2}), which has been proven in the special case $\tilde r =r-1$ and $p=\infty$ in the survey \cite{BaladiTsujii08b} by Baladi and Tsujii (inequality (15) and a comment following inequality (19) of \cite{BaladiTsujii08b}).
Since it is a key estimate in the proof of Theorem \ref{thm: decomposition}, 
we give a full proof  in Appendix \ref{app:A}.
\begin{lem}\label{L^p estimate}
There exists a constant $C_{F,G}>0$ such that for any $p\in [1,\infty]$, $u\in L^p(\mathbb{R}^d)$, $n\geq 0$ and $\ell \not\hookrightarrow n$, the following holds:
If $\tilde r \leq r-1$, then
\begin{equation}\label{ineq: (15) of the paper}
\Vert \Delta _n \mathbb{L} _{F,G} \tilde \Delta _\ell u\Vert_{L^p} \le C _{F,G}  2^{-\tilde r \max\{n,\ell\}}\Vert u\Vert_{L^p}.
\end{equation}
If  $\tilde r \geq  1$, 
 then 
\begin{equation}\label{ineq: (15) of the paper 2}
\Vert \Delta _n \mathbb{L} _{F,G} \tilde \Delta _\ell u\Vert_{L^p} \le C _{F,G}  2^{\min \{ n,\ell \}-\tilde r \max\{n,\ell\}}\Vert u\Vert_{L^p}.
\end{equation}
\end{lem}
%
Fix $\sigma  \in (0,s)$ satisfying that
\begin{equation}\label{eq:conditionsigma}
\sigma >s - \tilde r +1 + \delta  \quad  \text{when $\tilde r>1$}
\end{equation} 
with some $\delta \in (0, \tilde r-1)$.
Let $\gamma _{n, \ell}=1 $ when $\tilde r\leq 1$ and $\gamma _{n,\ell } =2^{\min \{ n,\ell \}}$ when $\tilde r >1$. Then, 
applying Lemma \ref{L^p estimate} together with \eqref{eq:above3}, 
we can find a nonnegative $\ell ^q$ sequence $\{c_\ell \}_{\ell \geq 0}$ satisfying $\Vert \{c_\ell \}_{\ell \geq 0}\Vert _{\ell ^q}\leq 1$ such that 
\begin{align*}
 \Vert \mathbb L_{1,n} 'u  \Vert _{L^p}   &\leq C_{F,G} \sum_{\ell : \ell \not\hookrightarrow n} \gamma _{n, \ell} 2^{-\tilde r \max\{ n,\ell \}} \Vert \Delta _\ell u\Vert_{L^p} \\
  & \leq C_{F,G} \sum_{\ell : \ell \not\hookrightarrow n} \gamma _{n, \ell} 2^{-\sigma \ell -\tilde r\max\{ n,\ell \}} c_\ell \Vert u\Vert_{B^\sigma_{pq}}.
\end{align*}

Therefore, in the case $q=\infty$, 
\begin{equation}
\sup _{n\geq 0} 2^{sn} \Vert \mathbb L_{1,n} 'u  \Vert _{L^p} 
 \leq
C_{F,G} \Vert u\Vert _{B^\sigma_{pq}} \sup _{n\geq 0}\sum_{\ell : \ell \not\hookrightarrow n}  \gamma _{n,\ell }2^{sn-\sigma \ell -\tilde r \max\{ n,\ell \}} .
\end{equation}
%
%
%
Since $-\tilde r\max \{ n,\ell \} \leq -\tilde rn$, and we assumed $\tilde r\geq s$ and $\sigma >0$,
\[
\sum_{\ell : \ell \not\hookrightarrow n} 2^{sn-\sigma\ell-\tilde r\max\{n,\ell\}} 
\le 
(1-2^{-\sigma} ) ^{-1} <\infty.
\]
Furthermore,  
\eqref{eq:conditionsigma} 
 implies that when $\tilde r>1$, we have 
\begin{equation}\label{eq:0430b}
\min \{ n,\ell \} + sn-\sigma \ell -\tilde r \max\{ n,\ell \} 
<
\begin{cases}
-(\tilde r -s)(n -\ell ) -\delta \ell  \quad &(\ell < n)\\
-(s+1)(\ell -n ) -\delta \ell \quad &( \ell \geq n)
\end{cases},
\end{equation}
which is bounded by $-\delta \ell$, and thus
\[
\sum_{\ell : \ell \not\hookrightarrow n}  2^{\min \{ n,\ell \} + sn-\sigma \ell -\tilde r \max\{ n,\ell \}} 
<(1-2^{-\delta } ) ^{-1} <\infty.
\]
Hence, we obtain \eqref{eq:0602b} with $\bar C_{F,G} =C_{F,G} (1-2^{-\sigma })^{-1}$.


When $q<\infty$, we can similarly have
\begin{equation*}
\Big( \sum _{n\geq 0} 2^{sqn} \Vert \mathbb L_{1,n} 'u  \Vert _{L^p} ^q \Big) ^{1/q} \le C_{F,G}\Vert u\Vert_{B^{\sigma}_{pq}}\Big( \sum_{n\ge 0}  \Big(  \sum_{\ell : \ell \not\hookrightarrow n} \gamma _{n,\ell }2^{sn-\sigma\ell-\tilde r\max\{n,\ell\}} c_\ell\Big)^q\Big)^{1/q}
\end{equation*}
for some $\{c_\ell\} _{\ell \geq 0}\in \ell^q$ satisfying $\Vert \{c_\ell\} _{\ell \geq 0} \Vert _{\ell ^q}\leq 1$.
Let $\tau = \tilde r -s$. 
Then, $\tau>0$ by assumption of Theorem \ref{thm: main}, 
and
$sn
 -\tilde r\max\{n,\ell\} \leq sn-\tilde rn \leq  -\tau (n -\ell ) $.  Moreover, it follows from \eqref{eq:0430b} that $\min \{ n, \ell \} + sn-\sigma \ell -\tilde r\max\{n,\ell\} \leq  -\tau (n -\ell )$ when $\tilde r>1$. Hence, this leads to
\begin{equation*}
 \sum_{\ell: \ell \not\hookrightarrow n} \gamma _{n,\ell } 2^{sn-\sigma\ell- \tilde r\max\{n,\ell\}} c_\ell
  \le   \sum_{\ell\ge 0} 1_{\{0 \not\hookrightarrow \cdot\}}(n-\ell) 2^{-\tau (n -\ell )} c_\ell.
\end{equation*}
Therefore, using an argument similar to the one following \eqref{eq: geological sequence} together with \eqref{eq:0602c}, we have
\begin{align}\label{eq:e33}
\Big( \sum _{n\geq 0} 2^{sqn} \Vert \mathbb L_{1,n} 'u  \Vert _{L^p} ^q \Big) ^{1/q} &\le C_{F,G}\Vert u\Vert_{B^{\sigma}_{pq}} \left( \sum _{\ell \in \Z : 0  \not\hookrightarrow \ell} 2^{-\tau \ell} \right)\\
&\leq \frac{ C_{F,G}2^{-\tau (3-\log _2\lambda )}}{1-2^{-\tau}}\Vert u\Vert_{B^{\sigma}_{pq}} .\notag
\end{align}
This completes the proof of \eqref{eq:0602b}.
\end{proof}

\begin{remark}\label{rmk:3}
We cannot show \eqref{eq:0602b} for $s=\tilde r$ and $q<\infty$ in the same manner as the case $q=\infty$. 
Indeed, in a manner similar to obtaining 
\eqref{eq:e33}, we have $\sum_{\ell : \ell \not\hookrightarrow n} 2^{\tilde rn-\sigma\ell-\tilde r\max\{n,\ell\}} \geq \sum_{\ell : \ell \not\hookrightarrow n, \ell \leq n} 2^{-\sigma\ell } \geq 1 $ for any sufficiently large $n\geq 0$. 
It follows that
$\Big\{\sum_{\ell : \ell \not\hookrightarrow n} 2^{\tilde rn-\sigma\ell-\tilde r\max\{n,\ell\}}\Big\} _{n\geq 0}$  is 
 not in $\ell^q$  ($q<\infty$).
\end{remark}

\subsection{Completion of the proof}
Now we can complete the proof of Theorem \ref{thm: main} by reducing the transfer operator $\mathcal L_{f^n,g^{(n)}} = \mathcal L^n_{f,g}$ to a family of transfer operators on the local charts  and applying Theorem \ref{thm: decomposition} for each $n\geq 1$.
\begin{proof}[Proof of Theorem \ref{thm: main}]
For the time being, we shall fix $n\geq 1$ and suppress it from the notation. 
Let $\iota$ be an  isometric embedding from $\mathscr{B}_{pq}^s(M)$ to $\bigoplus _{i=1}^I B_{pq}^s(U_i)$
(equipped with the norm $\Vert (u_i)_{i=1}^I\Vert _{\bigoplus _{i=1}^I B_{pq}^s(U_i)} =\sum _{i=1}^I\Vert u_i\Vert _{B_{pq}^s}$) 
\color{black}
given by
\[
\iota (u) =\left( \phi _i\cdot u\circ \kappa _i^{-1}\right)_{i=1}^I, \quad u\in \mathscr{B}_{pq}^s(M).
\]

For $1\leq i,j\leq I$, we let $U_{ij}\equiv U_{n,ij}=\kappa _j(V_j\cap (f^n)^{-1}(V_i))$  and let $F_{ij}\equiv F_{n,ij}$ be a $\Ci ^r$ mapping   
on $\R ^d$
such that
\[
F_{ij}(x) =\kappa _i \circ  f^n \circ \kappa _j^{-1}(x) \quad \text{for   $x\in U_{ij}$}.
\]
Furthermore, let 
$G_{ij}\equiv G_{n,ij}$ be 
 a compactly supported $\Ci^{\tilde r}$ function on $\R^d$ such that
\[
G_{ij} (x) =\phi _i\circ F_{ij}(x) \cdot \phi _j(x) \cdot g^{(n)}\circ \kappa _j^{-1}(x)  \quad \text{for   $x\in \R ^d$}.
\]
By the existence of Markov partitions for expanding maps \cite{FU2010}, one can find finitely many compact sets $\{ K_\ell \} _{\ell =1}^{N(f^n)}$ and their neighbourhoods $\{ V_\ell \} _{\ell =1}^{N(f^n)}$ satisfying \eqref{eq:re7} and \eqref{eq:re8}.
Finally we define $\mathbf L\equiv \mathbf L_n:\bigoplus _{i=1}^I B_{pq}^s(U_i) \to  \bigoplus _{i=1}^I B_{pq}^s(U_i)$ by
\begin{equation}\label{eq:charttransfer2}
\mathbf L\left( (u_i) _{i=1}^I\right) = \left( \sum _{j=1}^I \mathbb L_{ij}\color{black}u_j \right) _{i=1}^I \quad \text{for $(u_i) _{i=1}^I\in  \bigoplus _{i=1}^I B_{pq}^s(U_i)$},
\end{equation}
 where $ \mathbb L_{ij}\equiv \mathbb L_{n,ij}:B_{pq}^s(U_j) \to B_{pq}^s(U_i)$ for $1\leq i,j\leq I$
 is given  by
\begin{equation}\label{eq:charttransfer}
 \mathbb L_{ij} u(x) = 
 \sum _{F_{ij}(y)=x} G_{ij}(y) \cdot u(y)   \quad \text{for   $x\in \R ^d$}.
\end{equation}
(More precisely, we define $\mathbf L$ as a bounded operator on $\bigoplus _{i=1}^I \Ci ^s(U_i)$ by \eqref{eq:charttransfer2} and \eqref{eq:charttransfer}, and it can be extended to  a bounded operator on $\bigoplus _{i=1}^I B_{pq}^s(U_i)$ by virtue of Theorem \ref{thm: decomposition}.)
Then it is straightforward to see that
the following commutative diagram holds:
\begin{equation}\label{eq:f2}
\begin{array}{ccc}
\mathscr{B}_{pq}^s(M) & \stackrel{\mathcal L_{f,g}^n}{\to} & \mathscr{B}^s _{pq}(M) \\
\downarrow\text{\scriptsize{$\iota$}} & \circlearrowleft & \downarrow\text{\scriptsize{$\iota$}} \\
\bigoplus _{i=1}^I B^s_{pq}(U_i) & \stackrel{\mathbf L }{\to} & \bigoplus _{i=1}^I B^s_{pq}(U_i).
\end{array}
\end{equation}
In particular, $\Vert \mathcal L_{f,g}^n u\Vert _{\mathscr B^s_{pq}}=\Vert \mathbf L\circ \iota (u)\Vert _{\bigoplus _{i=1}^I B^s_{pq}(U_i)}$ for each $u\in \mathscr B^s_{pq}(M)$.

It follows from Theorem \ref{thm: decomposition} that one can find a constant $\hat C _n>0$ and $\sigma  \in (0,s)$ such that for each $1\leq i,j\leq I$ satisfying  $U_{ij}\neq \emptyset$
and
 each $u\in B^s_{pq}(U_j)$,  we get the 
 bound of $\Vert \mathbb L_{ij} u \Vert _{B^s_{pq}} $  by
\begin{multline}\label{eq:e1}
 \gamma _s\sup _{x\in F_{ij}(U_{ij})} \left\vert \sum _{F_{ij}(y)=x}  \frac{1}{\vert \det DF_{ij} (y)\vert } \right\vert  ^{1/p'}
 \left\Vert G_{ij}  \det DF_{ij}\right\Vert _{L^\infty}  \\
\times \left(\min_{x\in U_{ij}}\min _{\stackrel{v\in \R ^d}{\vert v\vert =1}}  \vert DF_{ij}(x)v\vert  \right)^{-s} \Vert u \Vert _{B^s_{pq}} + \hat C_n \Vert u \Vert _{B^\sigma _{pq}},
\end{multline}
where the second factor of the first term is replaced with $1$ when $p=1$.
On the other hand, it is straightforward  to see that there is a constant $C_2>0$ independently of $n,i,j$ such that the first term of \eqref{eq:e1} is bounded by
\begin{equation}\label{eq:f1}
C_2\gamma _s\Vert \mathcal L_{f,\vert \det Df\vert ^{-1}} ^n1_{M}\Vert _{L^\infty} ^{1/p'} R_n(g) \left( \min_{x\in M}\min _{\stackrel{v\in T_xM}{\vert v\vert =1}}  \vert Df^n(x)v\vert  \right)^{-s} \Vert u \Vert _{B^s_{pq}} .
\end{equation}
Moreover, it follows from  the estimate (4.5) in \cite{BKS96} that $\Vert \mathcal L_{f,\vert \det Df\vert ^{-1}} ^n1_{M}\Vert _{L^\infty}$ is bounded by a constant $C_3>1$ which is independent of $n$.

Therefore, by \eqref{eq:f2}, we obtain the bound of $\Vert \mathcal  L_{f, g} ^n u \Vert _{\mathscr B^s_{pq}} $  by 
\[
IC_2C_3^{1/p'}\gamma _s
\left( \min_{x\in M}\min _{\stackrel{v\in T_xM}{\vert v\vert =1}}  \vert Df^n(x)v\vert  \right)^{-s}
R_n(g)\Vert u \Vert _{\mathscr B^s_{pq}} +I \hat C _n\Vert u \Vert _{\mathscr B^\sigma_{pq}}
\]
for any $n\geq 1$.
Since the embedding $\mathscr B^s _{pq}(M)\hookrightarrow \mathscr B^\sigma_{pq}(M)$ is compact (see \cite[Corollary 2.96]{BCD}), we can apply  Hennion's theorem \cite{Hennion1993} (see also \cite[Appendix A]{Baladibook2}): we finally have 
\[
r_{\mathrm{ess}} (\mathcal L _{f,g}\vert _{\mathscr B^s_{pq} (M)}) \leq \liminf _{n\to \infty} \left( IC_2C_3^{1/p'}\gamma _s
\left( \min_{x\in M}\min _{\stackrel{v\in T_xM}{\vert v\vert =1}}  \vert Df^n(x)v\vert  \right)^{-s}
R_n(g) \right) ^{1/n}.
\]
This completes the proof of Theorem \ref{thm: main}.
\end{proof}

\appendix
\section{Proof of Lemma \ref{L^p estimate}}\label{app:A}
We follow \cites{BaladiTsujii08b,Baladibook2}. 
Recall  
$\{ V_j\} _{j=1}^{N(F)}$ in 
 \eqref{eq:re8}. 
 Let $\{ \theta _j \}_{j=1}^{N(F)}$ be a partition of unity subordinate to  $\{ V_j \} _{j=1} ^{N(F)}$.
 Given $1\leq j\leq N(F)$, let $I_j$ be the set of integers $i$ in $\{1,\dots,N(F)\}$ such that $F(V_i) \cap V_j \neq \emptyset$. 
 For $1\leq j\leq N(F)$, let $T_j :\R ^d \to \R^d$ be a $\Ci ^r$ mapping such that
\[
T_j  (x)= (F \vert _{V_j} )^{-1}(x) \quad  \text{for  $x \in F(V_j)$},
\]
that $T_j(x) \not \in V_j $ if $x\not \in F(V_j)$,
and that $\vert T_j(x) - T_j(y)\vert >c_j \vert x-y \vert $ for any $x,y \in \R ^d$ with some constant $c_j >0$. 
 Then we have
 \begin{equation}\label{eq:0603}
 \mathbb L_{F,G} u = \sum _{j=1}^{N(F)} \sum _{i\in I_j}    \theta _j  \cdot ( \theta _i \cdot G \cdot u) \circ T_i .
 \end{equation}
Indeed, we  have
\[
 \sum _{j=1}^{N(F)} \sum _{i\in I_j}  \theta _j  (x)   ( \theta _i \cdot G \cdot u) \circ T_i (x) 
 =  \sum _{j: x\in V_j} \theta _j  (x) \sum _{i\in I_j}     ( \theta _i \cdot G \cdot u) \circ T_i (x)  .
\]
If $x \not \in F(V_i)$, then $T_i(x)\not \in V_i$ and it yields $( \theta _i \cdot G \cdot u) \circ T_i (x) =0$, thus
\[
 \sum _{j: x\in V_j} \theta _j  (x) \sum _{i\in I_j}     ( \theta _i \cdot G \cdot u) \circ T_i (x) 
 =  \sum _{j: x\in V_j} \theta _j  (x) \sum _{i\in I_j: x\in F(V_i)}     ( \theta _i \cdot G \cdot u) \circ (F \vert _{V_i} )^{-1} (x)  .
\]
On the other hand, for $j$ such that $x\in V_j$, $x\in F(V_i) $ implies $i \in I_j$, so
\[
  \sum _{i\in I_j: x\in F(V_i)}     ( \theta _i \cdot G \cdot u) \circ (F \vert _{V_i} )^{-1} (x) 
  =  \sum _{i: x\in F(V_i)}     ( \theta _i \cdot G \cdot u) \circ (F \vert _{V_i} )^{-1} (x) .
\]
Furthermore, for $i$ such that $ x\in F(V_i)$, setting $y= (F \vert _{V_i} )^{-1} (x) $ gives $F(y) = x$ and $y\in 
 V_i$, therefore
 \[
\sum _{i: x\in F(V_i)}    ( \theta _i \cdot G \cdot u) \circ (F \vert _{V_i} )^{-1} (x) 
  = \sum _{F(y) =x} \sum _{i : y\in V_i}( \theta _i \cdot G \cdot u) (y) = \sum _{F(y) =x} ( G \cdot u) (y) .
\]
 Hence, we get \eqref{eq:0603}, and 
 it follows from Minkowski's inequality that
 \[
 \Vert \Delta _n \mathbb L_{F,G} \Delta _\ell  u \Vert _{L^p}  \leq 
 \sum _{j=1} ^{N(F) } \sum _{i\in I_j} \Vert  \Delta _n \mathbb L_{ij}  \Delta _\ell u \Vert _{L^p},
 \]
 where $\mathbb L_{ij}u =  [\theta _j  \cdot ( \theta _i \cdot G ) \circ T_i ] \cdot u\circ T_i$ with a $\Ci ^{\tilde r}$ function $\theta _j  \cdot ( \theta _i \cdot G ) \circ T_i$ and a $\Ci ^r$ mapping $T_i$.
 
 Therefore,  it suffices to show the following local version of Lemma \ref{L^p estimate}: 
Let $T:\R ^d \to \R ^d$ be a $\Ci ^r$ mapping such that $\vert T(x) - T(y)\vert >\tilde c \vert x-y \vert $ for any $x,y \in \R ^d$ with some constant $\tilde c >0$, and that $T: V\to T(V)$ is a $\Ci ^r$ diffeomorphism on a bounded open set 
 $V$.
Let $\tilde G $ be a $\Ci ^{\tilde r}$ function whose support is included in $V$. 
Define  (with slight abuse of notation) $\mathbb L _{T, \tilde G}:L^p (\R ^d ) \to L^p (\R ^d)$ by
\[
\mathbb L _{T, \tilde G}u = \tilde G \cdot u \circ T,\quad   u\in L^p (\R ^d ).
 \]
 Also, we use the notation $\ell \not\hookrightarrow n$ when it holds $2^n > \Lambda 2^{\ell+4}$, with
 \begin{align*}
 \Lambda=\max_{w\in\supp (\tilde G) }\max_{\stackrel{v\in \R^d}{\vert v\vert =1}} \vert DT^{tr}(w)v\vert ,
 \end{align*}
  where   $DT^{tr}(w)$ is the transpose matrix of $DT(w)$.
\begin{lem}\label{L^p estimate2}
There exists a constant $C_{T,\tilde G} >0$ such that for any $p\in [1,\infty]$, $u\in L^p(\R ^d )$, $n\geq 0$ and $\ell \not \hookrightarrow n$, 
it holds
\begin{align}\label{ineq: Applem1}
\Vert \Delta _n \mathbb{L} _{T,\tilde G} \tilde \Delta _\ell u\Vert_{L^p} \le C _{T,\tilde G} 2^{-\tilde r \max\{n,\ell\}}\Vert u\Vert_{L^p}
\end{align}
when $\tilde{r}\leq r- 1$ and
\begin{align}\label{ineq: Applem2}
\Vert \Delta _n \mathbb{L} _{T,\tilde G} \tilde \Delta _\ell u\Vert_{L^p} \le C _{T,\tilde G} 2^{\min\{n,\ell\}-\tilde r \max\{n,\ell\}}\Vert u\Vert_{L^p}
\end{align}
when $\tilde{r}\geq 1$.
\end{lem} 
\begin{proof} 
We start the proof by noticing that $\Delta _n \mathbb L_{T,\tilde G} \tilde\Delta _\ell :L^p(\R ^d) \to L^p(\R ^d)$ is the composition $H_n^\ell \circ \mathbb M_1$
of a bounded operator $\mathbb M_1 : \Phi\mapsto \Phi\circ T \cdot \vert \det DT\vert $ on $L^p(\R ^d)$, whose operator norm is bounded by a constant $C_T >0$ (see \eqref{eq:rev7}), and 
an integral operator $H_{n}^\ell :\Phi\mapsto \int _{\R ^d}V_n^\ell (\cdot ,y) \Phi(y) dy$ on $L^p(\R ^d)$ with kernel 
\begin{equation}\label{eq:revdef4}
V _n^\ell (x,y ) =\int _{\R ^{3d}} e^{i(x-w)\xi +i (T(w) -T(y))\eta } \tilde G(w) \psi _n (\xi ) \tilde \psi _\ell (\eta )dwd\xi d\eta .
\end{equation}
We will give a bound of $V_n^\ell$ by a convolution kernel together with the factor $2^{-\tilde r\max \{n,\ell \}}$: 
Let    $b: \R ^d \to (0,1]$   be the  integrable function given by
\begin{equation}\label{eq:b}
b(x)= 
\begin{cases}
1 \quad &(\vert x\vert \leq 1),\\
\vert x\vert ^{-d-1} \quad  & (\vert x\vert >1),
\end{cases} 
\end{equation}
and we set
\begin{equation}\label{eq:revdef5}
b_m: \R^d \to \R ,\quad b_m(x) = 2^{dm} \cdot b(2^{m} x) \quad \text{for $m>0$}.
\end{equation}
Then, we will show that there is a constant $\tilde C_{T,\tilde G}>0$ such that for all $\ell \not\hookrightarrow  n$,
\begin{equation}\label{eq:revkey}
\vert V _n^\ell (x,y )\vert \leq \tilde C_{T, \tilde  G} 2^{-\tilde r\max \{ n,\ell \}}  b_{\min \{ n,\ell \}} (x-y)\quad \mathrm{if}\ \tilde{r}\le r-1,
\end{equation} 
and
\begin{equation}\label{eq:revkey2}
\vert V _n^\ell (x,y )\vert \leq \tilde C_{T, \tilde  G} 2^{\min \{ n,\ell \} -\tilde r\max \{ n,\ell \}}  b_{\min \{ n,\ell \}} (x-y)\quad \mathrm{if}\ \tilde{r}>1.
\end{equation} 
It follows from \eqref{eq:revkey}, by  Young's inequality, that 
\begin{align*}
\Vert H_n^\ell  \Phi  \Vert _{L^p} &\leq  \tilde C _{T,\tilde G} 2^{-\tilde r\max \{ n,\ell \}} \Vert b_{\min \{ n,\ell \}}* \Phi \Vert _{L^p} \\
&\leq  \tilde C _{T,\tilde G} 2^{-\tilde r\max \{ n,\ell \}}  \Vert b_{\min \{ n,\ell \}} \Vert _{L^1} \Vert  \Phi \Vert _{L^p},
\end{align*}
which is bounded by $ \tilde C _{T,\tilde G} \Vert b\Vert _{L^1}2^{-\tilde r\max \{ n,\ell \}}  \Vert \Phi \Vert _{L^p}$ since $\Vert b_m\Vert _{L^1} =\Vert b\Vert _{L^1}$ for any $m\geq 1$, and we get \eqref{ineq: Applem1} with $C_{T,\tilde G} =C_T \tilde C_{T,\tilde G}\Vert b\Vert _{L^1}$.
Similarly, \eqref{ineq: Applem2} follows from  \eqref{eq:revkey2}.
In the rest of the proof, we are dedicated to showing \eqref{eq:revkey} and \eqref{eq:revkey2}.


\bf Step 1: \rm We first prove  \eqref{eq:revkey}. \rm 
We 
further 
divide the proof of  \eqref{eq:revkey} according to whether  $\tilde r$ is an integer.

\bf Case (1-a): $\tilde r$  is an integer. \rm 
We recall \emph{integrations  by parts} in $w$: for each $\Ci ^2$ function $\phi :\R ^d \to \R$ and a compactly supported $\Ci ^1$ function $h :\R ^d \to \R$ with 
$ \sum _{j=1}^d (\partial _j \phi )^2\neq 0$ on  the support of $h$, we set 
\begin{equation}\label{eq:hk}
h_k ^\phi (w) = \frac{i\partial _k\phi (w) \cdot h(w)}{ \sum _{j=1}^d (\partial _j \phi (w))^2} \quad k=1,\ldots ,d,
\end{equation}
belonging to $C^1_0(\R ^d)$.
Then we get
\begin{equation}\label{eq:integralbyparts}
\int e^{i\phi (w)} h(w) dw = -\sum _{k=1} ^d \int i\partial _k \phi (w) e^{i\phi (w)} \cdot h^\phi_k(w) dw  = \int  e^{i  \phi (w)} \cdot \sum _{k=1}^d \partial _k  h_k^\phi (w)   dw ,
\end{equation}
where $w=(w_k)_{k=1}^d \in \R^d$ and $\partial _k$ denotes partial differentiation with respect to $w_k$. 
We call this operation  \emph{integration  by parts for $h$ over $\phi$.} 
Define 
\begin{equation}\label{eq:revdef2}
\mathcal T_{\phi}h  =  \sum _{k=1}^d \partial _k h_k^\phi , \quad
\phi  _{( \xi ,\eta ,x,y)}(w) = (x-w)\xi +(T(w)-T(y))\eta .
\end{equation}
We simply write $\mathcal T_{ \xi ,\eta  }$ for $ \mathcal T_{\phi _{( \xi ,\eta ,x,y)} }$ (note that $ \mathcal T_{\phi _{( \xi ,\eta ,x,y)} }h$ does not depend on $x$ and $y$).
Then, we have
\begin{equation}\label{eq:rev14}
V^\ell _n(x,y) = \int e^{i\phi _{( \xi ,\eta ,x,y)}(w)}\mathcal T_{ \xi ,\eta  } ^{\tilde r}\tilde G(w) \psi _n(\xi ) \tilde \psi _\ell (\eta ) dwd\xi d\eta,
\end{equation}
by integrating \eqref{eq:revdef4} by parts for $\tilde G$ over $\phi _{( \xi ,\eta ,x,y)}$  $\tilde r$ times (that is possible by definition of $\tilde r$).

Put 
\[
 \tilde{ \mathcal G} _{n,\ell } (\xi ,\eta ,w )=  \mathcal G _{n,\ell } (2^n \xi ,2^\ell \eta ,w) \quad \text{with} \quad \mathcal G_{n,\ell } (\xi ,\eta ,w) = \mathcal T_{ \xi ,\eta  } ^{\tilde r _1}\tilde G (w) \psi _n (\xi )\tilde \psi _\ell (\eta ).
\]
Then, by  scaling we have
\[
\mathbb F^{-1} \mathcal G_{n,\ell } (u ,v ,w) =2^{dn +d\ell } \mathbb F^{-1} \tilde{ \mathcal G} _{n,\ell } (2^n u ,2^\ell v ,w),
\]
where $\mathbb F^{-1}$ is the inverse Fourier transform with respect to the variable $(\xi ,\eta )$.
Therefore, 
by \eqref{eq:revdef2} and \eqref{eq:rev14}, we have
\begin{align}\label{eq:rev18}
V^\ell _n(x,y) & = \int \mathbb F^{-1} \mathcal G_{n,\ell } (x-w, T(w) -T(y), w)dw\\
& = 2^{dn +d\ell }  \int \mathbb F^{-1} \tilde{ \mathcal G} _{n,\ell } (2^n (x-w) ,2^\ell (T(w) -T(y)) , w) dw.
\end{align}

Next we will see that  for any integers $k, k' \geq 0$, there is a constant $C_{k,k'} >0$  such that  
\begin{equation}\label{eq:re5}
\vert \mathbb F^{-1} \tilde{ \mathcal G} _{n,\ell } (u ,v ,w) \vert \leq \vert u \vert ^{-k} \vert v \vert ^{-k'} C_{k,k'} 2^{-\tilde r \max \{ n,\ell \}}
\end{equation}
for $(u ,v, w)\in \R^{3d}$.
This immediately follows if one can see that for any multi-indices $\alpha , \beta $, there is a constant $C_{\alpha ,\beta } >0$  such that  
\begin{equation}\label{eq:re3}
\vert \partial ^\alpha _\xi \partial _\eta ^\beta \tilde{ \mathcal G}_{n,\ell }(\xi  , \eta  , w ) \vert  \leq C_{\alpha ,\beta }2^{ -\tilde r\max \{n,\ell \}}
\end{equation}
for  $(\xi ,\eta ,w) \in \R^{3d}$, where  $\partial ^\alpha _\xi =\partial ^{\alpha _1}_{\xi _1} \partial ^{\alpha _2}_{\xi _2} \cdots \partial ^{\alpha _d}_{\xi _d}$ for
 $\alpha = (\alpha _1, \ldots ,\alpha _d) \in \{ 0,1, \ldots \} ^d$. 
(Just consider integration by parts with respect to $(\xi ,\eta)$ for
\[
\mathbb F^{-1} \tilde{\mathcal G}_{n,\ell }(u,v,w) = (2\pi )^{-2d} \int e^{i(u\xi +v\eta )}  \tilde{\mathcal G} _{n,\ell }(\xi, \eta ,w) d\xi d\eta,
\]
with noticing that the support of $\tilde{\mathcal G} _{n,\ell} (\cdot ,\cdot ,w)$ is included in $[-2, 2] ^2$.)  
Hence, we will show \eqref{eq:re3}.

We first note that, 
  by induction with respect to $\tilde r$,
\begin{equation}\label{eq:rev11b}
\mathcal T_{ \xi ,\eta  } ^{\tilde r }\tilde G (w) = 
\frac{\Theta  _{3\tilde r }  ( \xi ,\eta ,w)}{\vert \xi -DT^{tr}(w) \eta \vert ^{4 \tilde r} }, 
\end{equation}
with some function  $\Theta _{3 \tilde r }
\equiv \Theta  _{3 \tilde r  ,T,\tilde G}
$, which  is a polynomial function of degree $3 \tilde r $ in $\xi$ and $\eta$ whose coefficients are 
$\Ci ^0$
 in $w$.
Therefore, 
we have 
 (by induction again) that 
\[
\partial _\xi ^\alpha \partial _\eta ^\beta \mathcal T_{ \xi ,\eta  } ^{\tilde r }\tilde G (w) = 
\frac{\Theta ^{(1)} _{3\tilde r +\vert \alpha \vert + \vert \beta \vert }  ( \xi ,\eta ,w )}{\vert \xi -DT^{tr}(w) \eta \vert ^{4 \tilde r +2\vert \alpha \vert +2 \vert \beta \vert}}
\]
with some function  $\Theta ^{(1)} _{3 \tilde r +\vert \alpha \vert + \vert \beta \vert }$, 
 which  is a polynomial function of degree $3 \tilde r +\vert \alpha \vert +\vert \beta \vert$ in $\xi$ and $\eta$ whose coefficients are  $\Ci ^{0}$ in $w$,
 where $\vert \alpha \vert = \sum _j \alpha _j$  for 
each multi-index 
$\alpha = (\alpha _1, \ldots ,\alpha _d) \in \{ 0,1, \ldots \} ^d$. 
Notice also that by definition, there is an integer $N_1(T)>0$ such that  for each $ \xi \in \supp (\psi _n), \eta \in \supp (\tilde \psi _\ell )$ and $w \in \R ^d$,
\[
 \vert \xi -DT^{tr} (w) \eta \vert  \geq 2^{\max \{ n,\ell \} -N_1(T)} \quad \text{if $\ell \not \hookrightarrow n$}.
\]
Hence, 
for each multi-indices $\alpha , \beta$ and $( \xi ,\eta ,w) \in \mathbb R^{3d}$ satisfying $\psi _n(\xi)  \cdot \tilde \psi _\ell (\eta ) \neq 0$ with $\ell \not \hookrightarrow n$, we have
\begin{equation}\label{eq:rev11}
\vert \partial ^\alpha _\xi \partial _\eta ^\beta \mathcal T ^{\tilde r} _{\xi ,\eta }\tilde G (w) \vert \leq \tilde C_{\alpha ,\beta } 2^{-(\tilde r + \vert \alpha \vert  + \vert \beta \vert ) \max \{n,\ell \}}
\leq 
\tilde C_{\alpha ,\beta } 2^{-n \vert \alpha \vert -\ell \vert \beta \vert -\tilde r\max \{n,\ell \}}, 
\end{equation}
for some constant $\tilde C_{\alpha ,\beta } >0$.
Since $\Vert \partial _{\xi }^\alpha  \psi _n  \Vert _{L^\infty} \leq C_{\alpha }( \rho ) 2^{-n \vert \alpha \vert}$ with some constant $C_{\alpha }(\rho ) >0$ (recall \eqref{eq:0603b}), 
 it follows  that  for each multi-indices $\alpha , \beta$ and
 $( \xi ,\eta ,w) \in \mathbb R^{3d}$, 
\begin{equation}\label{eq:rev11c}
\vert \partial ^\alpha _\xi \partial _\eta ^\beta \mathcal G_{n,\ell } (w) \vert \leq C_{\alpha ,\beta } 2^{-n \vert \alpha \vert -\ell \vert \beta \vert -\tilde r\max \{n,\ell \}}, 
\end{equation}
for some constant $C_{\alpha ,\beta } >0$. 
%
%
%
Thus, we immediately obtain \eqref{eq:re3}. 

Combining \eqref{eq:rev18} and \eqref{eq:re5}, we get that $\vert V_n^\ell (x,y) \vert $ is bounded by
\begin{align}\label{eq:rev19}
C' 2^{dn+d\ell -\tilde r \max\{ n,\ell \}} \int _Kb(2^n(x-w)) b(2^\ell (T(w)-T(y)))dw,
\end{align}
with some constant $C'>0$ and $K=\supp \tilde G$.
If $\ell \leq n$, we have
\begin{align*}
2^{dn} \int _Kb (2^n (x-w)) b(2^\ell (T(w)-T(y)))dw \leq 2^{dn}\int b(2^n(x-w))dw =\int b(u) du
\end{align*}
by setting $u=2^n(x-w) $. 
Hence,  with $C'_1= C'\Vert b\Vert _{L^1}$, we get  
\[
\vert V_n^\ell (x,y) \vert \leq C_1' 2^{d\ell -\tilde r\max\{ n,\ell \}} \leq C_1' 2^{d\min \{ n,\ell \} -\tilde r\max\{n,\ell \}}.
\]
In the other case (when $\ell > n$), by setting $\tilde u=2^\ell (T(w) -T(y))$, in a similar manner one can see that 
$ \vert V_n^\ell (x,y) \vert \leq C_{2}' 2^{dn  -\tilde r\max\{ n,\ell \}} \leq C_{2}' 2^{d\min \{ n,\ell \} -\tilde r\max\{n,\ell \}}$
with some constant $C_{2}'\equiv C'_{2,T}>0$.

Recalling \eqref{eq:b}, the above estimates imply \eqref{eq:revkey} for $\vert x-y\vert \leq 2^{-\min \{n,\ell \}}$.
So, we now assume that $\vert x-y\vert > 2^{-\min \{n,\ell \}}$.
We consider the case $\ell \leq n$ (the other case can be dealt similarly, as above).  
Let $q_0=\lfloor -\log _2\vert x-y\vert \rfloor$,
so that either $\vert x-w\vert \geq 2^{-q_0-1}$ or $\vert w-y\vert \geq 2^{-q_0-1}$ holds for any $w\in \R ^d$.
Hence with the notation 
$K_1=\{ w\in K \mid \vert x-w\vert \geq 2^{-q_0-1}\} $ 
 and $K_2=\{ w\in K  \mid \vert w-y\vert \geq 2^{-q_0-1} \} $, noting that $w\in K_2$ implies $\vert T(w) -T(y)\vert \geq \tilde c 2^{-q_0-1}$,  we have the bound of $\int _K b (2^n (x-w)) b(2^\ell (T(w)-T(y)))dw $ by
\begin{align*}
&\int _{K_1} \vert 2^n (x-w)\vert ^{-d-1} b(2^\ell (T(w)-T(y)))dw \\
&\quad + \int _{K_2}b (2^n (x-w)) \vert \tilde c2^\ell (w-y ) \vert ^{-d-1}dw\\
 \leq  & C_{3}' (2^{-d\ell-(n-q_0-1) (d+1)} + 2^{-dn-(\ell -q_0-1) (d+1)}) \leq C_{4}' 2^{-dn}b(2^{-\ell } ( x-y))
\end{align*}
with some constants $C_{3}' \equiv C_{3,T,d}'>0$ and $C_{4}' \equiv C_{4,T,d}'>0$, which completes the proof of \eqref{eq:revkey} due to \eqref{eq:rev19}.


\bf Case (1-b): $\tilde r$  is not  an integer. \rm 
%
We recall \emph{regularised integration by parts} in $w$: for each $\Ci ^{1+\delta }$ function $\phi :\R ^d \to \R$ and a compactly supported $\Ci ^\delta$ function $h :\R ^d \to \R$, for $\delta \in (0,1)$, with 
$ \sum _{j=1}^d (\partial _j \phi )^2\neq 0$ on  the support of $h$, each $h_k^\phi $ given in \eqref{eq:hk} belongs to $C^\delta _0(\R ^d)$ for $k=1,\ldots ,d$. Let $h^\phi _{k,\epsilon} := h^\phi _k* v_\epsilon$ for $\epsilon >0$, where $v_\epsilon (x) = \epsilon ^{-d} v(x/\epsilon)$ with a  $\Ci ^\infty$ function   $v:\R ^d \to \R _+$  supported in the unit ball and satisfying $\int v (x) dx =1$.
There is a constant $C_v>0$, independent of $\phi$, such that for each $\Ci ^\delta $ function $u :\mathbb R ^d \to \mathbb R$ and each small $\epsilon >0$,
\begin{equation}\label{eq:04221}
\Vert \partial _k (u*v_\epsilon ) \Vert _{C^0} \leq C_v \epsilon ^{\delta -1} \Vert u \Vert _{C^\delta } , \quad \Vert u - u*v_\epsilon  \Vert _{C^0} \leq C_v  \epsilon ^{\delta} \Vert u \Vert _{C^\delta }.
\end{equation}
Finally, for every real number $L\geq 1$, it holds
\begin{multline}\label{eq:regularisedintegralbyparts}
\int e^{iL\phi (w)} h(w) dw = -\sum _{k=1} ^d \int i\partial _k \phi (w) e^{iL\phi (w)} \cdot h_k^\phi (w) dw \\
= \int  \frac{e^{i L \phi (w)}}{L} \cdot \sum _{k=1}^d \partial _k h_{k,\epsilon } ^\phi dw -\sum _{k=1} ^d \int i\partial _k \phi (w) e^{iL\phi (w)} \cdot (h_k^\phi (w)-h_{k,\epsilon} ^\phi (w)) dw.
\end{multline}
(Compare with \eqref{eq:integralbyparts}.) 
We call this operation  \emph{regularised integration by parts for $h$ over $\phi$}.
Define 
\begin{equation}\label{eq:04221b2}
\mathcal T_{\phi }^{( 0,L)}h = L^{-1}\sum _{k=1}^d \partial _k h_{k,L^{-1}}^\phi , \quad
\mathcal T_{\phi }^{( 1,L)}h =  -i\sum _{k=1}^d \partial _k \phi  \cdot (h_k^\phi -  h_{k,L^{-1}}^\phi ) ,
\end{equation}
and denote $\mathcal T_{\phi _{(\xi,\eta ,x,y)}}^{( 0,L)}$ and $\mathcal T_{\phi _{(\xi,\eta ,x,y)} }^{( 1,L)}$ by $\mathcal T_{\xi ,\eta }^{( 0,L)}$ and $\mathcal T_{\xi ,\eta }^{( 1,L)}$, respectively.
Then, by integration of \eqref{eq:revdef4} by parts for $\tilde{G}$ over $\phi _{(\xi ,\eta ,x,y)}$ $[\tilde r]$ times and regularised integration by parts over $\phi _{(\xi ,\eta ,x,y)}$  one time with $\delta =\tilde r -[\tilde r]$ and $\epsilon = L ^{-1}$  (note  that 
$\phi _{(\xi, \eta ,x, y) } (w) = L \phi _{(\epsilon \xi , \epsilon \eta ,x, y) } (w)$), 
 we get
\begin{multline}\label{eq:rev20}
V^\ell _n(x,y) = \int e^{i\phi _{( \xi ,\eta ,x,y)}(w)} \mathcal T_{\epsilon \xi ,\epsilon \eta  } ^{(0,L)}\mathcal T_{\xi ,\eta  } ^{[\tilde r]}\tilde{G}(w) \psi _n(\xi ) \tilde \psi _\ell (\eta ) dwd\xi d\eta\\
+ \int e^{i\phi _{( \xi ,\eta ,x,y)}(w)} \mathcal T_{\epsilon \xi , \epsilon \eta  } ^{(1,L)}\mathcal T_{\xi ,\eta  } ^{[\tilde r]}\tilde{G}(w) \psi _n(\xi ) \tilde \psi _\ell (\eta ) dwd\xi d\eta ,
\end{multline}
which we denote by $V_{n,\ell} ^{(0,L)}(x,y) +V_{n,\ell }^{(1,L)}(x,y)$.

It follows from \eqref{eq:04221} and \eqref{eq:04221b2} that  (with  $L^{-1} =\epsilon $)
\begin{align}\label{eq:rev41}
\vert  \partial _\xi ^\alpha \partial _\eta ^\beta \mathcal T_{\epsilon \xi ,\epsilon \eta  } ^{(0,L)}h(w) \vert 
& \leq \epsilon \sum _{k=1} ^d  \left\vert \partial _k\left(    (\partial _\xi ^\alpha \partial _\eta ^\beta h^{\phi _{(\epsilon \xi ,\epsilon \eta ,x,y)}} _k) * v_\epsilon  \right)(w) \right\vert \\ 
&\leq C_v  \epsilon ^\delta   \sum _{k=1 } ^d \Vert \partial _\xi ^\alpha \partial _\eta ^\beta h^{\phi _{(\epsilon \xi ,\epsilon \eta ,x,y)}} _k  \Vert _{ C^\delta }.
\end{align}
In a similar manner, we get the bound of $\vert  \partial _\xi ^\alpha \partial _\eta ^\beta \mathcal T_{\epsilon \xi ,\epsilon \eta  } ^{(1,L)}h(w) \vert $ by
\begin{equation}\label{eq:rev41b}
C_v \epsilon ^\delta   \sum _{k=1 } ^d \sum _{(\alpha ^{(1)}, \beta ^{(1)} ,\alpha ^{(2)},\beta ^{(2)}) }\Vert \partial _\xi ^{\alpha ^{(1)}} \partial _\eta ^{\beta ^{(1)}} h^{\phi _{(\epsilon \xi ,\epsilon \eta ,x,y)}} _k  \Vert _{ C^\delta }\Vert \partial _\xi ^{\alpha ^{(2)}} \partial _\eta ^{\beta ^{(2)}} (\partial _k \phi _{(\epsilon \xi ,\epsilon \eta ,x,y)} )\Vert _{C^0},
\end{equation}
where the second summation is taken over multi-indices $(\alpha ^{(1)}, \beta ^{(1)} ,\alpha ^{(2)},\beta ^{(2)}) \in \{ 0,1, \ldots \} ^{4d}$ such that $\alpha ^{(1)} + \alpha ^{(2)} = \alpha$ and  $\beta ^{(1)} + \beta ^{(2)} =\beta$.
On the other hand,
\begin{equation}\label{eq:rev42}
\partial _k \phi _{(\epsilon \xi ,\epsilon \eta ,x,y)} = \epsilon \hat{\Theta } _1(w, \xi ,\eta), \quad h_k ^{\phi _{(\epsilon \xi ,\epsilon \eta ,x,y)}} (w) =  \frac{i\epsilon ^{-1} \hat{\Theta }_1 (w, \xi ,\eta ) h(w)}{\vert \xi -DT^{tr} (w) \eta \vert ^2} 
\end{equation}
where $\hat{\Theta }_1(w, \xi ,\eta ) $ is the $k$-th element of $-\xi +DT^{tr} (w) \eta$, which is a polynomial function of degree $1$ in $\xi$ and $\eta$
 whose coefficients are $\Ci ^{r-1}$ (at least $\Ci ^{\tilde r}$) 
 in $w$.
Therefore, 
applying \eqref{eq:rev41} and \eqref{eq:rev42} for $h= \mathcal T_{\xi ,\eta  } ^{[\tilde r]}\tilde{G}$, together with \eqref{eq:rev11}, 
we get
\[
\vert  \partial _\xi ^\alpha \partial _\eta ^\beta \mathcal T_{\epsilon \xi ,\epsilon \eta  } ^{(j,L)}\mathcal T_{\xi ,\eta  } ^{[\tilde r]}\tilde{G}(w) \vert \leq C_v\tilde C_{\alpha ,\beta } \epsilon ^{\delta } 2^{-n \vert \alpha \vert - \ell \vert \beta \vert - [\tilde r ]  \max \{n,\ell \}} \quad (j=0,1)
\]
for any $\xi \in \supp \psi _n$ and $\eta \in \supp \tilde{\psi } _\ell$ with $\ell \not \hookrightarrow n$, implying that, with $\epsilon =2^{-\max \{ n,\ell \}}$ and $\delta = \tilde r -[\tilde r]$, 
\[
\vert  \partial _\xi ^\alpha \partial _\eta ^\beta \mathcal T_{\epsilon \xi ,\epsilon \eta  } ^{(j,L)}\mathcal T_{\xi ,\eta  } ^{[\tilde r]}\tilde{G}(w) \vert \leq C_v\tilde C_{\alpha ,\beta } 2^{-n \vert \alpha \vert - \ell \vert \beta \vert - \tilde r \max \{n,\ell \}}\quad (j=0,1).
\]
%
Therefore, in a  manner similar to one in the case (1-a) (replacing $V_n^\ell $ with $V_{n,\ell} ^{(0,L)}$ and  $V_{n,\ell }^{(1,L)}$),  one can get \eqref{eq:revkey}. 

\bf Step 2: \rm Next we show \eqref{eq:revkey2}, according to whether $\tilde r$ is an integer.


\bf Case (2-a): $\tilde r$ is an integer. \rm
Recall $V^\ell _n$ from \eqref{eq:revdef4} and $b_m$ from \eqref{eq:revdef5}.
For each $(\xi, \eta ,x, y)$, define 
\begin{equation}\label{eq:revgoal1}
 \tilde \phi _{(\xi , \eta ,x,y)} (w) = (x-w)\xi + DT(y) (w-y) \eta ,\quad \mathcal R _{\phi _1, \phi _2}h = e^{i(\phi _1 - \phi _2)} h ,
\end{equation}
for continuous functions $\phi _1, \phi _2 ,h :\R ^d\to \R$.
We simply write $\mathcal R _0$ and $\mathcal R _0^{-1}$ for $\mathcal R_{\phi _1,\phi _2}$ and $\mathcal R_{\phi _2,\phi _1}$, respectively,  with $\phi _1= \phi _{(\xi , \eta ,x,y)}$ and $\phi _2=\tilde \phi _{(\xi , \eta ,x,y)}$.
Denote $\mathcal T_{\tilde \phi _{(\xi ,\eta ,x,y)}}$, $\mathcal T^{(0,L)}_{\tilde \phi _{(\xi ,\eta ,x,y)}}$ and $\mathcal T^{(1,L)}_{\tilde \phi _{(\xi ,\eta ,x,y)}}$ by $\tilde{ \mathcal T}_{\xi ,\eta ,y}$, $\tilde{ \mathcal T}^{(0,L)}_{\xi ,\eta ,y}$ and $\tilde{ \mathcal T}^{(1,L)}_{\xi ,\eta ,y}$, respectively.
Then   \eqref{eq:revdef4} can be rewritten  as 
\begin{equation}\label{eq:rev21}
V^\ell _n(x,y) =\int e^{i\tilde \phi _{(\xi , \eta ,x,y)} (w)} \mathcal R _0\tilde{G}(w) \psi _n(\xi ) \tilde \psi _\ell (\eta )dwd\xi d\eta.
\end{equation}
Note that  $\tilde \phi _{(\xi, \eta ,x, y)}:\R ^d \to \R$ is  of class $\Ci ^\infty$ and $\mathcal R_0\tilde{G}$ is $\Ci ^{\tilde r}$ because $\tilde{r}<r$,
so that
$
\tilde{\mathcal T}_{\xi, \eta , y}  \mathcal R_0 \tilde{G}: \R^d \to \R
$ 
 is well-defined  and of class
 $\Ci ^{\tilde r-1}$. 

Integrating \eqref{eq:rev21} by parts once on $w$ (for $\mathcal R_0\tilde{G}$ over $\tilde \phi _{(\xi ,\eta ,x,y)}$), we obtain
\begin{align}\label{eq:rev22}
V^\ell _n(x,y) &= \int  e^{i\tilde \phi _{(\xi , \eta ,x,y)} (w)}\tilde{\mathcal T}_{\xi, \eta , y}  \mathcal R_0 \tilde{G} (w)\psi _n(\xi ) \tilde \psi _\ell (\eta )dwd\xi d\eta\\
&= \int  e^{i\phi _{(\xi , \eta ,x,y)} (w)}\mathcal R_0^{-1}\tilde{\mathcal T}_{\xi, \eta , y}  \mathcal R_0 \tilde{G} (w) \psi _n(\xi ) \tilde \psi _\ell (\eta )dwd\xi d\eta.
\end{align}
Integrating \eqref{eq:rev22} by parts $\tilde r -1$ times  on $w$, we obtain
\[
V^\ell _n(x,y) = \int  e^{i\phi _{(\xi , \eta ,x,y)} (w)}\mathcal T ^{\tilde r -1} _{\xi ,\eta  } \mathcal R_0^{-1}\tilde{\mathcal T}_{\xi, \eta , y}  \mathcal R_0 \tilde{G} (w) \psi _n(\xi ) \tilde \psi _\ell (\eta )dwd\xi d\eta.
\]
In a  manner similar to the case (1-a), it can be shown that
\begin{equation}\label{eq:rev24}
\left\Vert \partial ^\alpha _\xi \partial _\eta ^\beta \mathcal T ^{\tilde r -1} _{\xi ,\eta  } \mathcal R_0^{-1}\tilde{\mathcal T}_{\xi, \eta , y}  \mathcal R_0 \tilde{G} \right\Vert _{L^\infty } \leq \tilde{C} _{\alpha ,\beta } 2^\ell 2^{-n \vert \alpha \vert -\ell \vert \beta \vert -\tilde r\max \{n,\ell \}},
\end{equation}
because one can have, by induction, that
\begin{align}\label{eq:rev11bb}
\mathcal T ^{\tilde r -1} _{\xi ,\eta  } \mathcal R_0^{-1}\tilde{\mathcal T}_{\xi, \eta , y}  \mathcal R_0 \tilde{G}
=
\frac{\Theta  ^{(2)}_{3(\tilde r-1)+1 }  ( \xi ,\eta ,w) + \Theta  ^{(3)}_{3(\tilde r-1)+1 }  ( \xi ,\eta ,w) \eta }{\vert \xi -DT^{tr}(w) \eta \vert ^{4 (\tilde r-1)}\vert \xi -DT^{tr}(y)\eta\vert^2 } , 
\end{align}
with some functions  $\Theta _{3 (\tilde r-1)+1 }^{(2)}$ and $\Theta _{3 (\tilde r-1)+1 }^{(3)}$, which  are  polynomial functions of degree $3 (\tilde r-1)+1 $ in $\xi$ and $\eta$
whose coefficients are  $\Ci ^0$. 

We considered $\mathcal T ^{\tilde r -1} _{\xi ,\eta  } \mathcal R_0^{-1}\tilde{\mathcal T}_{\xi, \eta , y}  \mathcal R_0$ instead of $\mathcal T^{\tilde r} _{\xi ,\eta}$ since $\mathcal T^{\tilde r} _{\xi ,\eta}$ is not well-defined when $\tilde r>r-1$.
The price we have to pay for this change  is the factor $\eta $ in \eqref{eq:rev11bb}, resulting in $2^\ell $ in \eqref{eq:rev24}. By definition of $\ell\not\hookrightarrow n$, $2^\ell$ is dominated  by $2^{\min\{n, \ell\}}$ multiplied by a constant.
Consequently, as in the case (1-a), one can see that  \eqref{eq:rev24} implies \eqref{eq:revkey2} .

\bf Case (2-b): $\tilde r$ is not an integer. \rm
%
We perform single integration by parts of \eqref{eq:revdef4} for $\tilde{G}$ over $\tilde{\phi} _{(\xi ,\eta ,x,y)}$, $[\tilde{r}]-1$ integration by parts over $\phi$, and a single regularised integration by parts over $\tilde \phi _{(\xi ,\eta ,x,y)}$ with $\delta =\tilde r-[\tilde r]$, resulting in
\begin{multline*}
V^\ell _n (x,y)
= \int e^{i\tilde \phi _{\xi, \eta , x, y}(w)} \tilde{\mathcal T}^{(0,L)}_{\xi /L, \eta /L ,y}\mathcal R_0\mathcal T^{[\tilde r]-1}_{\xi , \eta}\mathcal R_0^{-1} \tilde{\mathcal T}_{\xi , \eta  ,y}\mathcal R_0\tilde G(w) \psi _n(\xi ) \tilde{\psi }_\ell (\eta ) dwd\xi d\eta\\
+ \int e^{i\tilde \phi _{\xi, \eta , x, y}(w)} \tilde{\mathcal T}^{(1,L)}_{\xi /L, \eta /L ,y}\mathcal R_0\mathcal T^{[\tilde r]-1}_{\xi , \eta}\mathcal R_0^{-1} \tilde{\mathcal T}_{\xi , \eta  ,y}\mathcal R_0\tilde G(w) \psi _n(\xi ) \tilde{\psi }_\ell (\eta ) dwd\xi d\eta .
\end{multline*}
Then in view of the cases (1-b) and (2-a), we obtain \eqref{eq:revkey2}.
\end{proof}

\section*{Acknowledgements}
This work was partially supported by JSPS KAKENHI Grant Numbers 16J03963 and 17K05283. The second author shows his deep gratitude to the members of Faculty of  Engineering in Kitami Institute of Technology for their warm hospitality when he visited there in August 2016 and October 2017.
The authors also would like to express their gratitude to Masato Tsujii  for many fruitful discussions. 
Finally, the authors are deeply grateful to an anonymous reviewer for many important suggestions,
all of which substantially improved the paper.

\bibliographystyle{plain}
\bibliography{MATHabrv,NS}

\begin{bibdiv}
\begin{biblist}

\bib{AD2001}{article}{
      author={Aaronson, Jon},
      author={Denker, Manfred},
       title={{Local limit theorems for partial sums of stationary sequences
  generated by Gibbs--Markov maps}},
        date={2001},
     journal={Stochastics and Dynamics},
      volume={1},
      number={02},
       pages={193\ndash 237},
}

\bib{BCD}{book}{
      author={Bahouri, Hajer},
      author={Chemin, Jean-Yves},
      author={Danchin, Rapha{\"e}l},
       title={Fourier analysis and nonlinear partial differential equations},
   publisher={Springer Science \& Business Media},
        date={2011},
      volume={343},
}

\bib{BaillifBaladi}{article}{
      author={Baillif, Mathieu},
      author={Baladi, Viviane},
       title={{Kneading determinants and spectra of transfer operators in
  higher dimensions: the isotropic case}},
        date={2005},
     journal={Ergodic Theory Dyn. Syst.},
      volume={25},
       pages={1437\ndash 1470},
}

\bib{Baladi}{book}{
      author={Baladi, Viviane},
       title={Positive transfer operators and decay of correlations},
   publisher={World Scientific},
        date={2000},
}

\bib{Baladi2017}{article}{
      author={Baladi, Viviane},
       title={{The quest for the ultimate anisotropic Banach space}},
        date={2017},
     journal={Journal of Statistical Physics},
      volume={166},
      number={3-4},
       pages={525\ndash 557},
}

\bib{Baladi2018}{article}{
      author={Baladi, Viviane},
       title={Characteristic functions as bounded multipliers on anisotropic
  spaces},
        date={2018},
     journal={Proceedings of the American Mathematical Society},
      volume={146},
      number={10},
       pages={4405\ndash 4420},
}

\bib{Baladi2017correction}{article}{
      author={Baladi, Viviane},
       title={{Correction to: The Quest for the Ultimate Anisotropic Banach
  Space}},
        date={2018},
     journal={Journal of Statistical Physics},
      volume={170},
      number={6},
       pages={1242\ndash 1247},
}

\bib{Baladibook2}{book}{
      author={Baladi, Viviane},
       title={Dynamical zeta functions and dynamical determinants for
  hyperbolic maps},
      series={Ergebnisse der Mathematik und ihrer Grenzgebiete},
   publisher={Springer},
        date={2018},
}

\bib{BKS96}{article}{
      author={Baladi, Viviane},
      author={Kondah, Abdelaziz},
      author={Schmitt, Bernard},
       title={Random correlations for small perturbations of expanding maps},
        date={1996},
     journal={Random Comput. Dynam.},
      volume={4},
       pages={179\ndash 204},
}

\bib{BaladiTsujii07}{article}{
      author={Baladi, Viviane},
      author={Tsujii, Masato},
       title={{Anisotropic H{\"o}lder and Sobolev spaces for hyperbolic
  diffeomorphisms}},
        date={2007},
     journal={Ann. Inst. Fourier},
      volume={57},
       pages={127\ndash 154},
}

\bib{BT2008a}{collection}{
      author={Baladi, Viviane},
      author={Tsujii, Masato},
      editor={Burns, Keith},
      editor={Dolgopyat, Dmitry},
      editor={Pesin, Yakov},
       title={Dynamical determinants and spectrum for hyperbolic
  diffeomorphisms},
      series={Contemp. Math.},
   publisher={American Mathematical Society},
        date={2008},
      volume={469},
}

\bib{BaladiTsujii08b}{collection}{
      author={Baladi, Viviane},
      author={Tsujii, Masato},
      editor={Albeverio, S.},
      editor={Marcolli, M.},
      editor={Paycha, S.},
      editor={Plazas, J.},
      editor={Diederich, K.},
       title={Spectra of differential hyperbolic maps},
      series={Aspects Math.},
   publisher={Friedr. Vieweg},
        date={2008},
      volume={E38},
}

\bib{BY93}{article}{
      author={Baladi, Viviane},
      author={Young, Lai-Sang},
       title={On the spectra of randomly perturbed expanding maps},
        date={1993},
     journal={Comm. Math. Phys.},
      volume={156},
      number={2},
       pages={355\ndash 385},
}

\bib{BY94}{article}{
      author={Baladi, Viviane},
      author={Young, Lai-Sang},
       title={{Erratum: ``On the spectra of randomly perturbed expanding
  maps''}},
        date={1994},
     journal={Comm. Math. Phys.},
      volume={166},
      number={1},
       pages={219\ndash 220},
}

\bib{Faure}{article}{
      author={Faure, Fr{\'e}d{\'e}ric},
       title={Semiclassical origin of the spectral gap for transfer operators
  of a partially expanding map},
        date={2011},
     journal={Nonlinearity},
      volume={24},
       pages={1473\ndash 1498},
}

\bib{Gouezel2015}{collection}{
      author={Gou{\"e}zel, S{\'e}bastien},
      editor={Dolgopyat, Dmitry},
      editor={Pesin, Ya},
      editor={Pollicott, Mark},
      editor={Stoyanov, Li},
       title={Limit theorems in dynamical systems using the spectral method},
      series={Proceedings of Symposia in Pure Mathematics},
   publisher={American Mathematical Society},
        date={2015},
      volume={89},
}

\bib{GL2003}{article}{
      author={Gundlach, V.~M.},
      author={Latushkin, Yu},
       title={{A sharp formula for the essential spectral radius of the Ruelle
  transfer operator on smooth and H{\"o}lder spaces}},
        date={2003},
     journal={Ergodic Theory and Dynamical Systems},
      volume={23},
      number={1},
       pages={175\ndash 191},
}

\bib{Hennion1993}{article}{
      author={Hennion, Hubert},
       title={Sur un th{\'e}oreme spectral et son application aux noyaux
  lipchitziens},
        date={1993},
     journal={Proceedings of the American Mathematical Society},
      volume={118},
      number={2},
       pages={627\ndash 634},
}

\bib{KH95}{book}{
      author={Katok, Anatole},
      author={Hasselblatt, Boris},
       title={Introduction to the modern theory of dynamical systems},
   publisher={Cambridge University Press},
        date={1995},
}

\bib{Mane2012}{book}{
      author={Man{\'e}, Ricardo},
       title={Ergodic theory and differentiable dynamics},
   publisher={Springer Science \& Business Media},
        date={2012},
      volume={8},
}

\bib{FU2010}{book}{
      author={Przytycki, Feliks},
      author={Urba{\'n}ski, Mariusz},
       title={Conformal fractals: ergodic theory methods},
   publisher={Cambridge University Press},
        date={2010},
      volume={371},
}

\bib{Ruellebook}{book}{
      author={Ruelle, David},
       title={Thermodynamic formalism: the mathematical structures of classical
  equilibrium statistical mechanics},
   publisher={Addison-Wesley},
        date={1978},
}

\bib{Ruelle89}{article}{
      author={Ruelle, David},
       title={The thermodynamic formalism for expanding maps},
        date={1989},
     journal={Comm. Math. Phys.},
      volume={125},
       pages={239\ndash 262},
}

\bib{Ruelle1989}{article}{
      author={Ruelle, David},
       title={{Une extension de la th{\'e}orie de Fredholm}},
        date={1989},
     journal={Comptes rendus de l'Acad{\'e}mie des sciences. S{\'e}rie 1,
  Math{\'e}matique},
      volume={309},
      number={6},
       pages={309\ndash 310},
}

\bib{Ruelle1990}{article}{
      author={Ruelle, David},
       title={{An extension of the theory of Fredholm determinants}},
        date={1990},
     journal={Publications Math{\'e}matiques de l'Institut des Hautes
  {\'E}tudes Scientifiques},
      volume={72},
      number={1},
       pages={175\ndash 193},
}

\bib{Stein70}{book}{
      author={Stein, Elias~M},
       title={Singular integrals and differentiability properties of
  functions},
   publisher={Princeton university press},
        date={1970},
      volume={2},
}

\bib{Taylor}{book}{
      author={Taylor, Michael~E},
       title={{Pseudodifferential operators and nonlinear PDE}},
      series={Progress in Mathematics},
   publisher={Birkh{\"a}user Boston, Inc., Boston, MA},
        date={1991},
      volume={100},
}

\bib{Thomine2011}{article}{
      author={Thomine, Damien},
       title={A spectral gap for transfer operators of piecewise expanding
  maps},
        date={2011},
     journal={Discrete \& Continuous Dynamical Systems-A},
      volume={30},
      number={3},
       pages={917\ndash 944},
}

\bib{Triebel1973}{article}{
      author={Triebel, Hans},
       title={{Spaces of distributions of Besov type on Euclidean $n$-space.
  Duality, interpolation}},
        date={1973},
     journal={Arkiv f{\"o}r Matematik},
      volume={11},
      number={1-2},
       pages={13\ndash 64},
}

\bib{Triebel86}{article}{
      author={Triebel, Hans},
       title={{Spaces of Besov-Hardy-Sobolev type on complete Riemannian
  manifolds}},
        date={1986},
     journal={Arkiv f{\"o}r Matematik},
      volume={24},
      number={1},
       pages={299\ndash 337},
}

\bib{Viana2014}{book}{
      author={Viana, Marcelo},
       title={{Lectures on Lyapunov exponents}},
   publisher={Cambridge University Press},
        date={2014},
      volume={145},
}

\bib{Walters1982}{book}{
      author={Walters, Peter},
       title={An introduction to ergodic theory},
   publisher={Springer-Verlag},
        date={1982},
}

\end{biblist}
\end{bibdiv}

\end{document}